\documentclass{article}
\usepackage{authblk}

\usepackage{graphicx}

\usepackage{amsmath, amssymb, amsthm}
\usepackage{mathdots}
\usepackage{tikz-cd}
\usetikzlibrary{patterns, fit}

\usepackage[colorinlistoftodos]{todonotes}
\usepackage[colorlinks=true, allcolors=blue]{hyperref}

\newtheorem{theorem}{Theorem}[section]
\newtheorem{lemma}[theorem]{Lemma}
\newtheorem{corollary}[theorem]{Corollary}

\theoremstyle{definition}
\newtheorem{definition}[theorem]{Definition}

\newtheorem{remark}[theorem]{Remark}

\numberwithin{equation}{section}

\usepackage{thm-restate}

\usepackage[e]{esvect}
\newcommand{\Af}[1]{\vv*{A}{#1}}
\newcommand{\Gf}[1]{\vv*{G}{#1}}

\newcommand{\Afn}{\Af{n}}
\DeclareRobustCommand{\robustAfn}{\Af{n}}

\newcommand{\intv}{\mathbb{I}}

\newcommand{\vect}{\mathrm{vect}}
\newcommand{\id}{\mathrm{id}}

\newcommand{\smat}[1]{ \left[\begin{smallmatrix} #1 \end{smallmatrix}\right] }
\newcommand{\suchthat}{\ifnum\currentgrouptype=16 \;\middle|\;\else\mid\fi}

\newcommand{\fun}[2]{\left[#1,#2\right]}

\usepackage{mathtools}
\newcommand{\defeq}{\coloneqq}

\DeclareMathOperator{\rep}{rep}

\DeclareMathOperator{\End}{End}

\DeclareMathOperator{\Hom}{Hom}

\DeclareMathOperator{\supp}{\mathrm{supp}}
\DeclareMathOperator{\Ker}{Ker}
\DeclareMathOperator{\Coim}{Coim}
\DeclareMathOperator{\Ima}{Im}

\makeatletter
\def\Int #1{\expandafter\Int@i#1\@nil}
\def\Int@i #1,#2\@nil{{#1{:}#2}}

\def\IntC #1{\expandafter\IntC@i#1\@nil}
\def\IntC@i #1,#2\@nil{{#1{,}#2}}

\def\itoi #1{\expandafter\itoi@i#1\@nil}
\def\itoi@i #1,#2,#3,#4\@nil{_{\Int{#1,#2}}^{\Int{#3,#4}}}
\makeatother

\def\clap#1{\hbox to 0pt{\hss#1\hss}}
\def\mathllap{\mathpalette\mathllapinternal}
\def\mathrlap{\mathpalette\mathrlapinternal}
\def\mathclap{\mathpalette\mathclapinternal}
\def\mathllapinternal#1#2{%
  \llap{$\mathsurround=0pt#1{#2}$}}
\def\mathrlapinternal#1#2{%
  \rlap{$\mathsurround=0pt#1{#2}$}}
\def\mathclapinternal#1#2{%
  \clap{$\mathsurround=0pt#1{#2}$}}

\newcommand{\ardrawer}[3]{
  {
  \pgfmathsetmacro{\sidel}{#1};
  \pgfmathsetmacro{\n}{#2};
  \draw[rotate=45,#3] (0,0) -- (-\sidel,0) -- (-\sidel,\sidel);
  \pgfmathsetmacro{\stepsize}{\sidel / \n};
  \foreach \x in {1,...,\n} {
    \pgfmathsetmacro{\xx}{\x * \stepsize};
    \draw[rotate=45,#3] (-\xx,0) -- (-\xx, \xx);
    \draw[rotate=45,#3] (-\xx,\xx) -- (-\sidel, \xx);
  }
  }
}

\newcommand{\arsimple}[2]{
  \pgfmathsetmacro{\x}{#1};
  \pgfmathsetmacro{\xx}{2*\x};
  \filldraw[#2] (0,0) -- (-\x,-\x) -- (-\xx,0);
}

\makeatletter
\let\save@mathaccent\mathaccent
\newcommand*\if@single[3]{%
  \setbox0\hbox{${\mathaccent"0362{#1}}^H$}%
  \setbox2\hbox{${\mathaccent"0362{\kern0pt#1}}^H$}%
  \ifdim\ht0=\ht2 #3\else #2\fi
  }
\newcommand*\rel@kern[1]{\kern#1\dimexpr\macc@kerna}
\newcommand*\widebar[1]{\@ifnextchar^{{\wide@bar{#1}{0}}}{\wide@bar{#1}{1}}}
\newcommand*\wide@bar[2]{\if@single{#1}{\wide@bar@{#1}{#2}{1}}{\wide@bar@{#1}{#2}{2}}}
\newcommand*\wide@bar@[3]{%
  \begingroup
  \def\mathaccent##1##2{%
    \let\mathaccent\save@mathaccent
    \if#32 \let\macc@nucleus\first@char \fi
    \setbox\z@\hbox{$\macc@style{\macc@nucleus}_{}$}%
    \setbox\tw@\hbox{$\macc@style{\macc@nucleus}{}_{}$}%
    \dimen@\wd\tw@
    \advance\dimen@-\wd\z@
    \divide\dimen@ 3
    \@tempdima\wd\tw@
    \advance\@tempdima-\scriptspace
    \divide\@tempdima 10
    \advance\dimen@-\@tempdima
    \ifdim\dimen@>\z@ \dimen@0pt\fi
    \rel@kern{0.6}\kern-\dimen@
    \if#31
      \overline{\rel@kern{-0.6}\kern\dimen@\macc@nucleus\rel@kern{0.4}\kern\dimen@}%
      \advance\dimen@0.4\dimexpr\macc@kerna
      \let\final@kern#2%
      \ifdim\dimen@<\z@ \let\final@kern1\fi
      \if\final@kern1 \kern-\dimen@\fi
    \else
      \overline{\rel@kern{-0.6}\kern\dimen@#1}%
    \fi
  }%
  \macc@depth\@ne
  \let\math@bgroup\@empty \let\math@egroup\macc@set@skewchar
  \mathsurround\z@ \frozen@everymath{\mathgroup\macc@group\relax}%
  \macc@set@skewchar\relax
  \let\mathaccentV\macc@nested@a
  \if#31
    \macc@nested@a\relax111{#1}%
  \else
    \def\gobble@till@marker##1\endmarker{}%
    \futurelet\first@char\gobble@till@marker#1\endmarker
    \ifcat\noexpand\first@char A\else
      \def\first@char{}%
    \fi
    \macc@nested@a\relax111{\first@char}%
  \fi
  \endgroup
}
\makeatother
\title{Every 1D Persistence Module is a Restriction of Some Indecomposable 2D Persistence Module}

\author[1]{Micka\"{e}l Buchet}
\affil[1]{Institute of Geometry, TU Graz}
\affil[1]{buchet@tugraz.at}
\author[2]{Emerson G. Escolar}
\affil[2]{Center for Advanced Intelligence Project, RIKEN / Kyoto University Institute for Advanced Study}
\affil[2]{emerson.escolar@riken.jp}

\begin{document}
\maketitle
\begin{abstract}
  A recent work by Lesnick and Wright proposed a visualisation of $2$D persistence modules by using their restrictions onto lines, giving a family of $1$D persistence modules. We give a constructive proof that any $1$D persistence module with finite support can be found as a restriction of some indecomposable $2$D persistence module with finite support. As consequences of our construction, we are able to exhibit indecomposable $2$D persistence modules whose support has holes as well as an indecomposable $2$D persistence module containing all $1$D persistence modules with finite support as line restrictions. Finally, we also show that any finite-rectangle-decomposable $n$D persistence module can be found as a restriction of some indecomposable $(n+1)$D persistence module.
\end{abstract}

\noindent\textbf{Keywords} Representation theory, Multidimensional persistence\\
\textbf{Mathematics Subject Classification (2010)} 16G20, 55N99

\section{Introduction}

In the theory of persistent homology \cite{edelsbrunner2000topological}, $1$D persistence modules can be summarized and easily visualized using the so-called persistence diagrams, which led to successful applications of topological data analysis.
From a mathematical standpoint, the existence of the persistence diagram is a direct consequence of the fact that, viewed as representations of the underlying quiver $\mathbb{A}_n$, $1$D persistence modules can be uniquely decomposed into indecomposable representations that are intervals \cite{gabriel1972unzerlegbare}. The endpoints of the intervals give the birth and death indices of the topological features.

In the multidimensional persistence \cite{carlsson2009theory}, multidimensional persistence modules can be studied as representations of an underlying quiver which is a commutative $n$D grid. The indecomposable representations are no longer intervals and, more devastatingly, cannot be easily listed. More precisely, there is no complete discrete invariant that can describe all indecomposable $n$D persistence modules when $n \geq 2$ \cite{carlsson2009theory}. In the representation-theoretic language, the (large enough) commutative grid is of wild representation type.

Lesnick and Wright~\cite{lesnick2015interactive} proposed an interactive visualisation of $2$D persistence modules using restrictions onto lines on the grid. For each line, the restriction is a $1$D persistence module and thus can be summarized by a persistence diagram.
They took the approach of exploring restrictions onto all possible lines of non-negative slope in order to visualize and obtain insight into the structure of a $2$D persistence module.
Note that considering the barcodes over all such lines is equivalent to considering the rank invariant \cite{carlsson2009theory} of the $2$D persistence module.

For most of this work, we only  consider persistence modules with finite support. For the sake of simplicity, unless stated otherwise, all persistence modules will have finite support. The only exception is the infinite $2$D persistence module we build in Section~\ref{sec:lotr}.

In this work, we look into the possible outcomes of the restriction onto a single line of a $2$D persistence module that is indecomposable. It was hoped that, being the building blocks of $2$D persistence modules, the $2$D indecomposables should have restrictions that are ``simple'' or coming from a restricted set of possibilities. However, we show that this is not the case.

Our main result is a constructive proof that \emph{any} $1$D persistence module $V$ can be obtained via the restriction onto a line of an \emph{indecomposable} $2$D persistence module (of large enough, but finite, support). We show the following Theorem (see Sec.~\ref{sec:background} for precise definitions).
\begin{restatable}{theorem}{mainthm}
  \label{thm:main}
  Let $V \in \rep{\Afn}$. Then there exists an indecomposable $M \in \rep \Gf{}$ with finite support such that $V$ is a line restriction of $M$.
\end{restatable}

Theorem~\ref{thm:main} can be seen as another expression of how complicated the indecomposable $2$D persistence modules are. Stated another way, the indecomposable $2$D persistence modules collectively contain all possible $1$D persistence modules via line restrictions.
Another, less rigorous, interpretation is that by just looking at one line restriction $V$ of some $2$D persistence module $M$ (not necessarily indecomposable), and without any other information about the original $2$D persistence module $M$, we cannot infer much about the decomposition structure of $M$.
  More precisely, for each decomposition
  $
    V = \bigoplus_{i=1}^l V_i
    $
    (with $V_i$ not necessarily indecomposable), it is possible that
  $
    M \cong X \oplus \bigoplus_{i=1}^l M_i
  $
  such that $M_i$ are indecomposable, $M_i$ restricts to $V_i$, and $X$ restricts to $0$ on the line. Note that our construction of an indecomposable is not the only choice of such an $M_i$.

  Using our construction for Theorem~\ref{thm:main}, we are able to build several objects of particular interest.
First, we are able to exhibit indecomposable $2$D persistence modules whose support can have an arbitrary finite number of holes. We believe that their existence was not obvious and that providing a concrete example will help deepen the understanding of $2$D persistence modules. We also address the question of what is the smallest number of vertices needed in the support of an indecomposable $2$D persistence module with at least one hole.
Second, we build a single indecomposable $2$D persistence module with infinite support that contains all possible $1$D persistence modules with finite support as line restrictions.

Finally, we extend Theorem~\ref{thm:main} to finite-rectangle-decomposable $n$D persistence modules. In particular, we show that any finite-rectangle-decomposable $n$D persistence module $V$ can be obtained via a ``hyperplane restriction'' of an indecomposable $(n+1)$D persistence module. The construction directly allows us to build indecomposable $(n+1)$D persistence modules with an arbitrary finite number of $n$D holes.

We introduce basic definitions in Section~\ref{sec:background} and tools in Section~\ref{sec:tools}.
We then describe our construction of an indecomposable $2$D persistence module given $V$ and its properties in Section~\ref{sec:results}. In Section~\ref{sec:zoology}, we exhibit examples of $2$D indecomposables with holey support obtained as a result of our construction. There, we also find the smallest number of vertices needed in the support of an indecomposable $2$D persistence module with a hole, via an ad hoc construction. In Section~\ref{sec:lotr}, we build the indecomposable $2$D persistence module with infinite support containing all possible $1$D persistence modules with finite support as line restrictions. In Section~\ref{sec:rectangles} we discuss the extension to finite-rectangle-decomposable $n$D persistence modules.

%%% Local Variables:
%%% mode: latex
%%% TeX-master: "arxiv_main"
%%% End:

\section{Background}\label{sec:background}

We use the formalism of both category theory and representation theory, while providing clear intuition as much as possible. For convenience, we mainly use the representation theory of posets, but the representation theory of bound quivers is also used.
If needed, we refer the reader to~\cite{assem2006elements} for more details on the representation theory of bound quivers, and to~\cite{mac2013categories} for the category theory.

\subsection{The poset $\robustAfn$ and $1$D persistence modules}
A poset (a partially ordered set) $(P, \leq)$ can be viewed as a category with objects $i \in P$, and morphisms determined as follows. There is a unique morphism $i\rightarrow j$ if and only if $i\leq j$. We abuse notation and write $i\leq j$ to also mean this unique morphism.
Composition follows from transitivity: if there are morphisms $i\rightarrow j$ and $j\rightarrow k$, then $i\leq j$ and $j \leq k$, implying $i\leq k$. Thus, there is a unique morphism $i \rightarrow k$, which is the composition of the two morphisms. The identity of object $i\in P$ is clearly the morphism $i \leq i$.

% In this work, we use the language of representations of posets instead of quivers because posets have commutativity relations ``built-in'' automatically by the uniqueness of morphisms $i\rightarrow k$ (if existing).

We define $[n]$ to be the set $\{0,1,\hdots,n-1\}$. By $\Afn$, we mean the poset category of $[n] \defeq \{0,1,\hdots,n-1\}$ with the usual order $\leq$. We shall see below that the category $\Afn$ is intimately related to the quiver
\begin{equation}
  \label{eq:qn}
  \vec{\mathbb{A}}_n:
  \begin{tikzcd}
    0 \rar & 1 \rar & \hdots \rar & n-1
  \end{tikzcd}
\end{equation}
which can be thought of as the Hasse diagram of $\Afn$.

Let $K$ be a field, which we fix throughout this work. Recall that a $K$-category is a category where the set of morphisms between any two objects has $K$-vector space structure, and composition is $K$-bilinear.

Let $P$ be a poset category and $C$ a $K$-category. A $C$-valued representation of $P$ is a functor $F:P\rightarrow C$. That is, it assigns an object $F(i)$ of $C$ to each $i\in P$, and to each morphism $i\rightarrow j$ in $P$ (i.e. denoted by $i\leq j$ in the sequel) a morphism $F(i\leq j) : F(i) \rightarrow F(j)$ in $C$, satisfying the following properties: $F(i\leq i)$ is the identity on $F(i)$, and $F(j\leq k)F(i\leq j) = F(i\leq k)$.

A morphism between two $C$-valued representations of $P$ is nothing but a natural transformation. That is, given $F: P\rightarrow C$ and $G:P\rightarrow C$, a morphism $\eta: F \rightarrow G$ is a family $\{\eta_i: F(i) \rightarrow G(i)\}_{i\in P}$ of morphisms in $C$ such that
\[
  \begin{tikzcd}
    F(i) \rar{F(i\leq j)}\dar{\eta_i} & F(j)\dar{\eta_j} \\
    G(i) \rar{G(i\leq j)} & G(j)
  \end{tikzcd}
\]
commutes for all $i\leq j$ in $P$. The set of morphisms from $F$ to $G$ is denoted by $\Hom(F,G)$. Morphisms from $F$ to itself are called \emph{endomorphisms}, and we use the notation $\End(F) \doteq \Hom(F,F)$. Composition
$\begin{tikzcd}
  F \rar{\eta} & G \rar{\nu} & H
\end{tikzcd}
$
is the obvious one: $\nu\eta$ is defined by $(\nu \eta)_i = \nu_i \eta_i$. The category of $C$-valued representations of $P$ shall be denoted by $\fun{P}{C}$.

For the case $C = \vect_K$, the category of finite-dimensional $K$-vector spaces, we use the notation $\rep{P} \defeq \fun{P}{\vect_K}$ for the category of $\vect_K$-valued representations, also called $K$-linear representations. Unless specified otherwise, \emph{representation} shall mean a $K$-linear representation. The \emph{support} of $V \in \rep P$ is the set $\{i \in P \suchthat V(i) \neq 0\}$.

It is known that $\rep{P}$ is an additive $K$-category. The direct sum of two representations $V$ and $W$ is given by $(V\oplus W)(p) = V(p) \oplus W(p)$ for all $p\in P$. The zero representation is the functor $0$ that takes everything in $P$ to zero vector spaces and zero maps. A representation $V$ is said to be \emph{indecomposable} if $V \cong W \oplus W'$ implies $W$ or $W'$ is $0$.

For example, a representation $V$ of $\Afn$ is an assignment of the following data: to each object $i \in [n]$, a finite-dimensional $K$-vector space $V(i)$; to each pair $i\leq j$, a $K$-linear map $V(i\leq j) : V(i) \rightarrow V(j)$ so that the composition
\[
  V(i) \xrightarrow{V(i\leq j)} V(j) \xrightarrow{V(j\leq k)} V(k)
\]
is equal to $V(i\leq k) : V(i) \rightarrow V(k)$ for each $i\leq j \leq k$, and $V(i\leq i)$ is the identity for each $i$.

Thus, $V \in \rep\Afn$ is completely determined by vector spaces $V(i)$ for $i\in [n]$ and linear maps $V(i\leq i+1)$ for $i\in [n-1]$. This corresponds to the usual notion of a representation of the quiver $\vec{\mathbb{A}}_n$ given in Diagram~\eqref{eq:qn}. In fact, $\rep \Af{n} \cong \rep \vec{\mathbb{A}}_n$, where the latter is the category of finite-dimensional $K$-linear representations of the quiver $\vec{\mathbb{A}}_n$. We shall freely use this identification. A $1$D \emph{persistence module} is simply a representation of $\Afn$.

A fundamental result (Krull-Schmidt Theorem) is that every representation in $\rep\Afn$ can be uniquely decomposed into a finite direct sum of indecomposable representations, up to isomorphism and permutation of terms.
Furthermore, it is known \cite{gabriel1972unzerlegbare} that any indecomposable representation of $\Afn$ is isomorphic to an interval representation $\intv[i,j]$ for some $i\leq j \in [n]$. The \emph{interval representation} $\intv[i,j]$ is defined
by:
\[
  \intv[i,j](\ell) =
  \left\{
    \begin{array}{cc}
      K & \text{if } i\leq \ell \leq j, \\
      0 & \text{otherwise,}
    \end{array}
  \right.
  \text{ and }
  \intv[i,j](\ell \leq k) =
  \left\{
    \begin{array}{cl}
      1_K & \text{if } i\leq \ell \leq k \leq j, \\
      0 & \text{otherwise.}
    \end{array}
  \right.
\]

Combining the above two facts, we get that for each $V \in \rep \Afn$,
\[
  V \cong \bigoplus_{i = 1}^N \intv[b_i, d_i]^{m_i}
\]
uniquely for some positive integers $m_1,\hdots,m_N$ and $b_i \leq d_i \in [n]$ for $i \in \{1,\hdots, N\}$ with all pairs $(b_i,d_i)$ distinct. In other words, the $1$D persistence module $V$ is, up to isomorphism, entirely determined by the multiset consisting of pairs $(b_i,d_i)$ with multiplicities $m_i$ for all $i\in \{1,\hdots, N\}$. This can be drawn as a multiset of points in $\mathbb{R}^2$ and is called the \emph{persistence diagram} of the $1$D persistence module $V$.

In practice, the persistence module is often obtained as the sequence of homology vector spaces and induced homology maps of a sublevel set filtration parametrized over $\mathbb{R}$. In this case, the underlying poset is $(\mathbb{R},\leq)$ instead. However, some mild tameness assumptions on the filtration or persistence module guarantee the existence of a persistence diagram \cite{chazal2016structure,crawley2015decomposition}. Moreover, if there are only a finite number of ``critical points'', then we can consider the persistence module as a representation of $\Afn$.

\subsection{Commutative grids and line restrictions}

Given $V \in \rep{\Afn}$, we will construct an indecomposable persistence module over an equioriented commutative $2$D grid of size $w \times h$. The size of the grid will depend on $V$ itself, and so may be larger than $n$. To avoid having to specify sizes in our definition of line restrictions, we use an infinite discrete commutative grid.

First, we extend $\Afn$ to an infinite line by the inclusion of $\Af{n} = (\{0,\hdots,n-1\},\leq)$ in the poset $(\mathbb{Z},\leq)$. Thus, $\rep\Af{n}$ is a full subcategory of $\rep(\mathbb{Z},\leq)$. We use the following definition for the equioriented commutative $2$D grids.
\begin{definition}[Equioriented commutative $2$D grids]
  \leavevmode
  \begin{enumerate}
  \item We define $\Gf{}$ to be the poset category $(\mathbb{Z}\times\mathbb{Z}, \leq)$ with $(a,b) \leq (x,y)$ if and only if $a \leq x$ and $b \leq y$.
  \item $\Gf{h,w}$ is the poset category $([h]\times[w],\leq)$, which is the full subcategory of $\Gf{}$ with set of objects $[h]\times[w]$.
  \end{enumerate}
\end{definition}

To see why there is a ``commutative'' in the name, we provide as an example $\Gf{2,2}$. A representation $V \in \rep(\Gf{2,2})$ is completely determined by the data:
\[
  \begin{tikzcd}[column sep=6em]
    V((1,0)) \rar{\delta\defeq V((1,0)\leq(1,1))} & V((1,1)) \\
    V((0,0)) \rar{\alpha\defeq V((0,0)\leq(0,1))} \uar{\gamma \defeq V((0,0)\leq(1,0))} &
    V((0,1)) \uar[swap]{\beta\defeq V((0,1)\leq(1,1))}
  \end{tikzcd}
\]
where by functoriality,
\[
  \beta\alpha =  V((0,0)\leq(1,1)) = \delta \gamma.
\]
That is, a commutativity relation is automatically satisfied.

 We call representations of $\Gf{}$ as \emph{$2$D persistence modules}. If $V \in \rep\Gf{}$ has finite total dimension, then it has finite support, and can be viewed (after translation) as a representation of some finite $2$D grid $\Gf{h,w}$.

% In this language, representations of $\Gf{h,w}$ are the $2$D persistence modules.
% In this work, we use both terms interchangeably.
Alternatively, $\rep\Gf{h,w}$ is isomorphic to the representation category of the $h\times w$ equioriented $2$D grid quiver bound by commutativity relations. See for example, \cite{thin-decomposability}, for precise definitions.

\begin{definition}[Line]
  A \emph{line} $L$ is a functor
  \[
    L: (\mathbb{Z},\leq) \rightarrow (\mathbb{Z}\times\mathbb{Z},\leq)
  \]
  such that if $i \neq j$, then $L(i) \neq L(j)$.
\end{definition}
Note that the definition above does not correspond exactly with a geometric ``line''. For example, the functor $L$ defined on objects $x\in\mathbb{Z}$ by $L(x) = (2x,2x)$ is a line. However, it geometrically intersects points on the $2$D grid that are not part of its image, for example the point $(1,1)$. Our result does not rely on this peculiarity.

Given a line $L$, we get an induced functor between representation categories via composition:
\[
	\begin{array}{rlccl}
    	R_L &:& \rep \Gf{} & \rightarrow &\rep \mathbb{Z} \\
                & & (M: \Gf{} \rightarrow \vect_K) & \mapsto & (M\circ L: \mathbb{Z} \rightarrow \vect_K)
    \end{array}
\]
and with effect on morphisms also by composition. That is, if $\eta:M\rightarrow N$ is a morphism between representations of $\Gf{}$ (a natural transformation), then $(R_L(\eta))_i = \eta_{L(i)}$ defines the morphism $R_L(\eta) : ML \rightarrow NL$.

\begin{definition}[Line Restriction]
  \label{defn:line_restriction}
  Let $V \in \rep\Af{n} \subset \rep \mathbb{Z}$. We say that $V$ is a \emph{line restriction} of $M\in\rep\Gf{}$ if there is a line $L$ such that $R_L(M) \cong V$.
\end{definition}

%%% Local Variables:
%%% mode: latex
%%% TeX-master: "arxiv_main"
%%% End:

\section{Tools}
\label{sec:tools}

\subsection{Stacking}

To prove Theorem~\ref{thm:main}, for each $V \in \rep\Afn$ we construct an indecomposable $2$D persistence module $I$ such that $V$ is a line restriction of $I$. In order to easily construct $2$D persistence modules, we stack together representations of $\Af{w}$ (for some longer $w \geq n$). This stacking procedure can be theoretically justified by the following lemma.

\begin{restatable}[Stacking]{lemma}{stackinglem}  
  \label{lem:vertical_stacking}
  There is an isomorphism of categories
  \[
    \rep\Gf{h,w} = \fun{\Gf{h,w}}{\vect_K} \cong \fun{\Af{h}}{\fun{\Af{w}}{\vect_K}} = \fun{\Af{h}}{\rep\Af{w}}.
    % \fun{\Gf{h,w}}{\vect_K} \cong \fun{\Af{h}}{\fun{\Af{w}}{\vect_K}}.
  \]
\end{restatable}
In words, representations of the commutative grid $\Gf{h,w}$ can be viewed as $\rep\Af{w}$-valued representations of a $1$D grid (the vertical line $\Af{h}$) by collapsing each row to an object of $\rep\Af{w}$. In this interpretation, the representations of $\Af{w}$ are drawn as rows, and $h$ of them are stacked in a line, with a morphism from the $i$th row to the $j$th row for $i\leq j \in[h]$.

This can also be seen as ``currying'' a functor in $\rep \Gf{h,w}$. More generally, this is analogous to tensor-hom adjunction in $\vect_K$-enriched functor categories, as in Section~2.3 of \cite{kelly1982basic}. Note that the poset categories $\Gf{h,w}$, $\Af{w}$, and $\Af{h}$ are not $\vect_K$-enriched, but we can consider instead their $K$-linearizations $K\Gf{h,w}$, $K\Af{w}$, and $K\Af{h}$. The analogy follows by considering $K$-linear functors from the $K$-linearizations instead of functors from the poset categories and noting that $K\Gf{h,w}\cong K\Af{h}\otimes K\Af{w}$. We provide an elementary proof of Lemma~\ref{lem:vertical_stacking} in Appendix~\ref{appendix}.

The stacking lemma is an obvious extension of an idea in the paper \cite{asashiba2018matrix}, where representations of commutative ladders, which are commutative $2$D grids of size $2\times w$, are viewed as morphisms in $\rep\Af{w}$. This is essentially stacking two representations of $\Af{w}$. 

Objects $M$ in the category $\fun{\Af{h}}{\rep\Af{w}}$ are $\rep(\Af{w})$-valued representations of $\Af{h}$. Thus, each arrow $i\leq j \in \Af{h}$ is given the data of a morphism $M(i\leq j):M(i) \rightarrow M(j)$ in $\rep(\Af{w})$ since $M$ is $\rep(\Af{w})$-valued, instead of a morphism between vector spaces as when $\vect_K$-valued. In the next subsection, we study morphisms in $\rep(\Af{w})$ and provide a convenient way to describe them.

\subsection{Matrix Formalism}
\label{subsec:matrix_formalism}
Let us quickly review the matrix formalism introduced in the paper \cite{asashiba2018matrix}, which allows us to write a morphism $\phi:V\rightarrow W$ in $\rep\Af{w}$ in a matrix form. In this work, we only need to consider $V$ and $W$ equal to a direct sum of the interval representations, and this simplifies the presentation below.

Let $\phi:V\rightarrow W$ be a morphism in $\rep{\Af{w}}$ with $V$ and $W$ given as below:
\[
  \begin{tikzcd}[column sep=1.5em]
    V =
    \bigoplus\limits_{i=1}^m \intv[b_i,d_i] \rar{\phi} &
    \bigoplus\limits_{j=1}^{m'} \intv[b'_j,d'_j] = W.
  \end{tikzcd}
\]
Note that the intervals $\intv[b_i,d_i]$ for $i=1,\hdots,m$ (respectively $\intv[b'_j,d'_j]$ for $j=1,\hdots,m'$) may contain duplicates. This is not a problem.

Then $\phi$ can be written in a matrix form $\Phi=\left[\Phi_{j,i}\right]$ relative to these decompositions, with entries given by $\Phi_{j,i} = \pi_j \phi \iota_i$:
\[
  \begin{tikzcd}[column sep=1.5em]
    \intv[b_i,d_i] \rar{\iota_i} &
    \bigoplus\limits_{i=1}^m \intv[b_i,d_i] \rar{\phi} &
    \bigoplus\limits_{j=1}^{m'} \intv[b'_j,d'_j] \rar{\pi_j} &
    \intv[b'_j,d'_j]
  \end{tikzcd}
\]
where $\iota_i$ and $\pi_j$ are the evident inclusions and projections, respectively.
Instead of a scalar, each entry $\Phi_{j,i}$ in the matrix form is a morphism from $\intv[b_i,d_i]$ to $\intv[b'_j,d'_j]$. The following lemma tells us the possible morphisms between intervals.

\begin{lemma}[Lemma 1 of \cite{asashiba2018matrix}]
  \label{lem:homdim}
  Let $\intv[a,b],\intv[c,d]$ be interval representations of $\Af{w}$.
  \begin{enumerate}
  \item The dimension of $\Hom(\intv[a, b], \intv[c, d])$ as a $K$-vector space is either $0$ or $1$.
  \item There exists a canonical basis $\left\{f\itoi{a,b,c,d}\right\}$ for each nonzero $\Hom(\intv[a, b], \intv[c, d])$ such that
    \[
      (f\itoi{a,b,c,d})_i = \left\{
        \begin{array}{ll}
          1_K: K \rightarrow K, & \text{if }i \in [a,b] \cap [c,d] \\
          0, & \text{otherwise.}
        \end{array}
      \right.
    \]
    \item  In fact,
    \[
      \Hom(\intv[a,b],\intv[c,d])=
      \begin{cases}
        K f\itoi{a,b,c,d}, & c \leq a \leq d \leq b, \\
        0, & \text{otherwise}.
      \end{cases}
    \]
  \end{enumerate}
\end{lemma}
Going back to our morphism $\Phi = \left[\Phi_{j,i}\right]$ in matrix form, the entry at $(j,i)$ is given by
\[
  \Phi_{j,i} = c_{j,i} f\itoi{b_i,d_i,b'_j,d'_j}
\]
for some scalar $c_{j,i}\in K$ if the dimension of $\Hom(\intv[b_i, d_i], \intv[b'_j, d'_j])$ is nonzero; otherwise $\Phi_{j,i}$ is always $0$. As a shorthand, we hide the morphisms $f\itoi{b_i,d_i,b'_j,d'_j}$ whenever we display the matrix form of morphisms in $\rep\Af{w}$.

\subsection{Auslander-Reiten quivers}
We give a quick review of Auslander-Reiten quivers, which provide a very convenient framework to visualize our construction. By definition, the vertices of the Auslander-Reiten quiver are the indecomposable representations, and arrows are determined by the irreducible morphisms between indecomposables.
See the book \cite{assem2006elements}, for example, for a more detailed discussion.

% The Auslander-Reiten quiver is usually defined for quivers posets such as $\Afn$, but by the identification $\rep \vec{\mathbb{A}}_n \cong \rep \Afn$, we can view $\Afn$ as the finite quiver $\vec{\mathbb{A}}_n$. It is this quiver's Auslander-Reiten quiver we are considering.

For our purposes, we only need to consider the Auslander-Reiten quiver $\Gamma(\Afn)$ of $\Afn$, which takes on a relatively simple form. It is given by
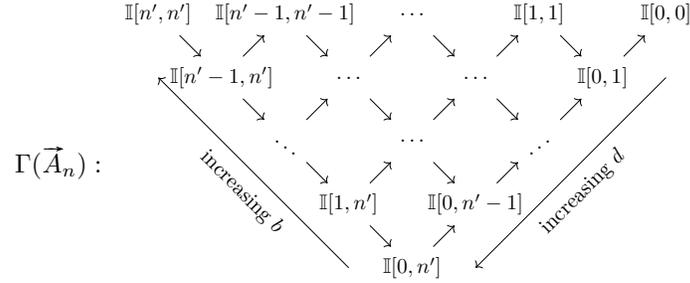
\[
  \newcommand{\mcp}{\mathclap}
  \Gamma(\Af{n}):
  \scalebox{0.8}{
    \begin{tikzpicture}[baseline=(A.center)]
      \matrix (A) [matrix of math nodes, nodes in empty cells, ampersand replacement=\&,
      nodes={minimum height=3em, minimum width=3em, anchor=center}, inner sep=0.5em, outer sep = 0.5em,column sep=0pt,row sep=0pt]
      {
        \mcp{\intv[n',n']} \&                    \& \mcp{\intv[n'-1,n'-1]} \&                  \& \hdots           \&                  \& \mcp{\intv[1,1]} \&                  \& \mcp{\intv[0,0]}      \\
                         \& \mcp{\intv[n'-1,n']} \&                      \& \hdots \&                  \& \hdots \&                  \& \mcp{\intv[0,1]} \&                       \\
                         \&                    \& \ddots               \&                  \& \hdots \&                  \& \iddots \&                  \&                       \\
                         \&                    \&                      \& \mcp{\intv[1,n']} \&                  \& \mcp{\intv[0,n'-1]} \&                  \&                  \&                       \\
                         \&                    \&                      \&                  \& \mcp{\intv[0,n']} \&                  \&                  \&                  \&                       \\
                         \&                    \&                      \&                  \&                  \&                  \&                  \&                  \&                       \\
      };
      % height y
      \foreach \y in {1,...,4} {
        \pgfmathtruncatemacro{\m}{5-\y};
        \foreach \l in {1,...,\m} {
          \pgfmathtruncatemacro{\x}{2*\l-1 + \y-1};
          \pgfmathtruncatemacro{\yu}{\y+1};
          \pgfmathtruncatemacro{\xu}{\x+1};
          \pgfmathtruncatemacro{\xuu}{\x+2};
          \path [<-] (A-\y-\x) edge (A-\yu-\xu);
          \path [<-] (A-\yu-\xu) edge (A-\y-\xuu);
        }
      }
      \path [->] (A-2-9.center) edge node[label={[rotate=45]below: increasing $d$}]{} (A-5-6.center);
      \path [->] (A-5-4.center) edge node[label={[rotate=-45]below: increasing $b$}] {} (A-2-1.center);
    \end{tikzpicture}
  }
\]
where $n' = n-1$ and where we adopt the convention of placing the simple representations (lifespan $0$ intervals) $\intv[b,b]$ on a horizontal line, with everything else below.

To visualize our construction, we shall abstractly represent the Auslander-Reiten quiver of $\Afn$ by a triangle:
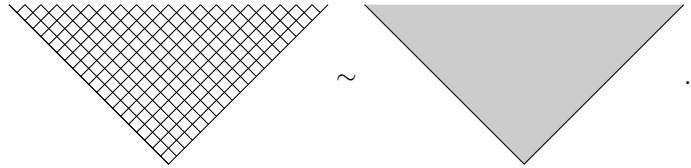
\[
  \begin{tikzpicture}[baseline=(current bounding box.center)]
    \ardrawer{3}{20}{use as bounding box}
  \end{tikzpicture}
  \sim
  \begin{tikzpicture}[baseline=(current bounding box.center)]
    \pgfmathsetmacro{\x}{3*sqrt(2)/2};

    \pgfmathsetmacro{\xx}{2*\x};
    \filldraw[use as bounding box,fill=black!20] (0,0) -- (-\x,-\x) -- (-\xx,0);
  \end{tikzpicture}.
\]
Furthermore, we will need to freely extend the length of the underlying poset to some $w \geq n$. We will thus consider the Auslander-Reiten quiver of $\Afn$ as being embedded inside that of $\Af{w}$, like so:
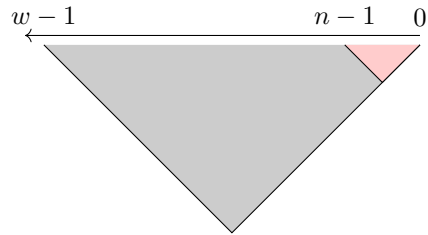
\[
  \begin{tikzpicture}[scale=0.5]
    % \ardrawer{6}{60}{}
    % \ardrawer{1}{10}{red}
    \arsimple{5}{fill=black!20}
    \arsimple{1}{fill=red!20}

    \draw [->] (0,0.25)node[above]{$0$} -- (-2, 0.25)node[above]{$n-1$} -- (-10,0.25)node[above]{$w-1$} -- (-10.5,0.25);
  \end{tikzpicture}
\]
where the top line keeps track of the simple representations, with $\intv[w-1,w-1]$ at the leftmost and $\intv[0,0]$ at the rightmost.

%%% Local Variables:
%%% mode: latex
%%% TeX-master: "arxiv_main"
%%% End:

\section{Main Construction}\label{sec:results}
Before proving our main result, we introduce the following concept, which is essential in our proof for indecomposability. This is an offshoot of the idea behind the algebraic construction in \cite{buchet_et_al:socg}, where certain configurations of indecomposables were found.

\begin{definition}[Vertical]
  A collection of interval representations $\{\intv[b_i,d_i]\}_{i=1}^m$ is said to be \emph{vertical} if and only if all the pairs $(b_i,d_i)$ are distinct and there exists a constant $\mu$ such that $\mu = \frac{b_i+d_i}{2}$ for all $i \in \{1,\hdots,m\}$.
\end{definition}

If a collection of intervals is vertical, then they are all located along a vertical line on the Auslander-Reiten quiver, as follows:
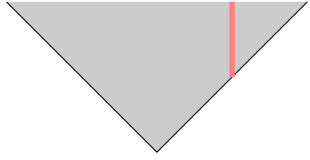
\[
  \begin{tikzpicture}
    \arsimple{2}{fill=black!20}
    \draw[red!50, line width=2pt] (-1,0) -- (-1,-1);
  \end{tikzpicture}
\]
because of the way we illustrate them. The following (almost trivial) observation about vertical intervals forms an essential part of our construction. Note that the matrix form below is written using the matrix formalism as reviewed in Subsection~\ref{subsec:matrix_formalism}.
\begin{lemma}
  \label{lem:vertical}
  Let $W = \bigoplus_{i=1}^m \intv[b_i,d_i]$ be a representation of ${\Af{w}}$ such that $\{\intv[b_i,d_i]\}_{i=1}^m$ is vertical. Then the matrix form of any morphism $\phi: W\rightarrow W$ is diagonal:
  \[
    \phi =
    \left[
      \begin{array}{ccc}
        c_1 & & \\
            & \ddots & \\
            & & c_m
      \end{array}
    \right]
  \]
  for some scalars $c_1,c_2,\hdots, c_m \in K$ (recall our shorthand of hiding the morphisms $f\itoi{b_i,d_i,b_j,d_j}$).
\end{lemma}
\begin{proof}
  Since $\{\intv[b_i,d_i]\}_{i=1}^m$ is vertical, there exists $\mu$ as in the definition and \emph{distinct} numbers $h_i$ such that $b_i = \mu - h_i$ and $d_i = \mu + h_i$ for all $i$.

  For $i\neq j$, it follows from part \textit{3} of Lemma~\ref{lem:homdim} that
  \[
    \Hom(\intv[b_i,d_i],\intv[b_j,d_j]) = 0
  \]
  since the inequality $b_j \leq b_i \leq d_j \leq d_i$ cannot be satisfied. Otherwise, we have
  \[
    \mu - h_j \leq \mu - h_i \leq \mu + h_j \leq \mu + h_i
  \]
  implying $h_j \geq h_i$ and $h_j \leq h_i$, and so $h_j = h_i$. This is a contradiction since $h_i$ and $h_j$ are supposed to be distinct. Thus, off-diagonal terms are always zero. Terms on the diagonal are some constant $c_i \in K$ times $f\itoi{b_i,d_i,b_i,d_i}$, which are identity maps.
  %\qed
\end{proof}

We reproduce here our Theorem~\ref{thm:main} and provide a proof.

\noindent\parbox{\textwidth}{\mainthm*}
\begin{proof}
Without loss of generality, we identify $V = \bigoplus_{i=1}^m \intv[b_i,d_i]$.
The proof takes the form of a construction process composed of three parts: separate-and-shift, verticalize, and coning.

\paragraph{Separate-and-shift} We separate our intervals into distinct death-indices with the extra condition that these indices are large enough so that the next step is possible.

  More precisely, we choose $m$ distinct integers $d'_i$ such that
  \begin{enumerate}
  \item $d'_i \geq d_i$ for all $i$ and
  \item $\frac{b_j + d'_j}{2} \leq d'_i$ for all pairs $i$ and $j$.
  \end{enumerate}
  The first condition is clearly possible to satisfy. For the second condition, choose distinct $d''_i$ satisfying the first, and let $\ell = \max_i(b_i + d''_i)$ and define $d'_i \defeq d''_i + \ell$. Then
  \[
    \frac{b_j + d'_j}{2} = \frac{b_j + d''_j}{2} + \frac{\ell}{2} \leq \ell \leq d''_i + \ell = d'_i
  \]
  for all pairs $i,j$, showing that both conditions can be satisfied by $m$ distinct integers $d_i'$.

  We then choose an integer $w \geq n$ such that $w-1 \geq d'_i$ for all $i$. For example, $w = \max\{n-1,\max\limits_i(d'_i)\} + 1$ will do. We have the inclusion $\Afn \subset \Af{w}$, so that we can consider $V \in \rep\Afn \hookrightarrow \rep\Af{w}$ as a representation of $\Af{w}$, essentially by padding $0$'s from vertices $n$ to $w-1$. Further constructions will take place in $\rep\Af{w}$.

  By construction, there exists a nonzero morphism $f\itoi{b_i,d'_i,b_i,d_i}: \intv[b_i, d'_i] \rightarrow \intv[b_i, d_i]$ for each $i \in \{1,\hdots,m\}$, which are in fact surjections. We form the representation
  $
  V' \defeq \bigoplus\limits_{i=1}^m \intv[b_i, d'_i]
  $
  and the morphism
  \begin{equation}
    \label{eq:separate_shift}
    \begin{tikzcd}[ampersand replacement=\&, column sep=large]
      V' \rar{\smat{1 & & \\ & \ddots & \\ & & 1}} \& V
    \end{tikzcd}
  \end{equation}
  (written in the matrix formalism). The diagonal entries can indeed be nonzero because of the existence of the nonzero morphisms.

  The matrix form given in Diagram~\eqref{eq:separate_shift} is not the identity matrix and does not represent the identity morphism. The entries are morphisms between intervals, and we are hiding the factors $f\itoi{b_i,d'_i,b_i,d_i}$ for simplicity.

  Pictorially, we think of $V = \bigoplus_{i=1}^m \intv[b_i,d_i]$ as a multiset of points on the Auslander-Reiten quiver $\Gamma(\Afn)$ of $\Afn$, which is considered as a subquiver of $\Gamma(\Af{w})$. Separate-and-shift moves each point (treating multiple copies of the same point as different) to a different death-index $d$. We illustrate this procedure as follows:
  
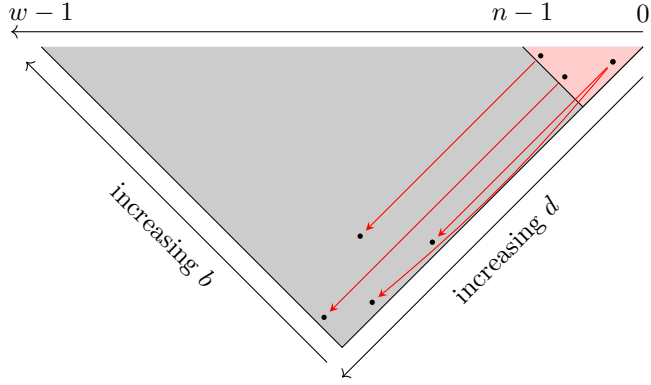
\[
    \begin{tikzpicture}[scale=0.8]
      \arsimple{5}{fill=black!20}
      \arsimple{1}{fill=red!20}

      \draw [->] (0,0.25)node[above]{$0$} -- (-2, 0.25)node[above]{$n-1$} -- (-10,0.25)node[above]{$w-1$} -- (-10.5,0.25);

      \begin{scope}[every node/.style={shape=circle,fill,minimum size=2pt,inner sep=0, outer sep=2pt}]
        \path {
          (-0.5,-0.25) node (a1) {}
          ++(-3,-3) node (a2) {}
        };
        \path[draw,red,-stealth] (a1) edge (a2);

        \path {
          (-0.5,-0.25) node (a1) {}
          ++(-4,-4) node (a2) {}
        };
        \path[draw,red, -stealth] (a1) edge[bend left=5] (a2);

        \path {
          (-1.7,-0.15) node (a1) {}
          ++(-3,-3) node (a2) {}
        };
        \path[draw,red, -stealth] (a1) edge (a2);

        \path {
          (-1.3,-0.5) node (a1) {}
          ++(-4,-4) node (a2) {}
        };
        \path[draw,red, -stealth] (a1) edge (a2);
      \end{scope}

      \path [->] (0,-0.5) edge node[label={[rotate=45]below: increasing $d$}]{} +(-5,-5);
      \path [->] (-5.25,-5.25) edge node[label={[rotate=-45]below: increasing $b$}]{} +(-5,+5);

    \end{tikzpicture}
  \]
  where the rightmost point appears with multiplicity $2$, and each copy is brought to a different death index.

\paragraph {Verticalize} We move our separated-and-shifted intervals into a vertical by increasing birth indices. The ``shift'' part (the second condition) of the previous step ensures that the increased birth indices do not go beyond the death indices.

Let $\mu = \max(\frac{b_i+d'_i}{2})$, the maximum of the midpoints of the intervals $\intv[b_i,d'_i]$. Note that by second condition of the previous step, $\mu \leq d'_i$ for all $i$. We define $b'_i \defeq 2 \mu - d'_i$ for all $i$, which can be visualized as ``pivoting'' $d'_i$ around $\mu$ to get $b'_i$. Note that while $\mu$ may be a half-integer, $b'_i$ so defined is always an integer.

It follows that
\begin{equation}
  \label{eq:b_condition}
 	b_i \leq b'_i \leq \mu \leq d'_j
\end{equation}
for all pairs $i,j$. This inequality is very important for our construction, and we will refer again to this fact later.

In particular, we see that $b'_i \leq d'_i$ for all $i$. Thus, the intervals $\intv[b'_i, d'_i]$ are non-empty and distinct (since $d'_i$ are), and $\{\intv[b'_i,d'_i]\}_{i=1}^m$ is vertical with common midpoint $\mu = \frac{b'_i + d'_i}{2}$ as given. Furthermore, Ineq.~\eqref{eq:b_condition} says that the common midpoint $\mu$ shared by all $\intv[b'_i,d'_i]$ serves as a wall separating all birth indices $b'_i$ and all death indices $d'_j$.

Since $b'_i \geq b_i$, there exists a nonzero
$f\itoi{b'_i,d'_i,b_i,d'_i}: \intv[b'_i, d'_i] \rightarrow \intv[b_i, d'_i]$ for each $i$, which are in fact inclusions. We form the representation
$
\widebar{V} \defeq \bigoplus\limits_{i=1}^m \intv[b'_i, d'_i]
$
and the morphism
\begin{equation}
  \label{eq:verticalize}
  \begin{tikzcd}[ampersand replacement=\&, column sep=large]
    \widebar{V} \rar{\smat{1 & & \\ & \ddots & \\ & & 1}} \& V'.
  \end{tikzcd}
\end{equation}

We illustrate this step by the following figure:
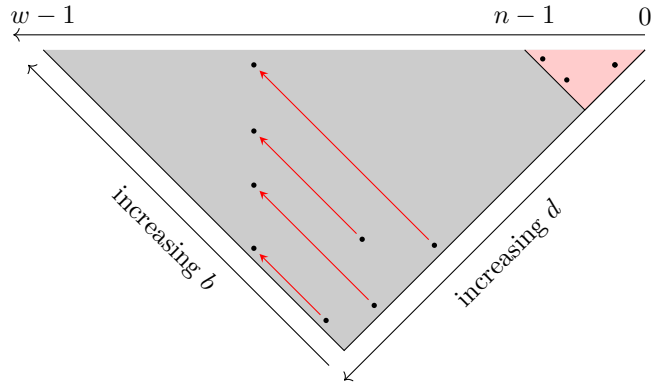
\[
  \begin{tikzpicture}[scale=0.8]
    \arsimple{5}{fill=black!20}
    \arsimple{1}{fill=red!20}

    \draw [->] (0,0.25)node[above]{$0$} -- (-2, 0.25)node[above]{$n-1$} -- (-10,0.25)node[above]{$w-1$} -- (-10.5,0.25);

    \begin{scope}[every node/.style={shape=circle,fill,minimum size=2pt,inner sep=0, outer sep=2pt}]
      \path {
        (-0.5,-0.25) node (a1) {}
        ++(-3,-3) node (a2) {}
        ++(-3,+3) node (a3) {}
      };
      \path[draw,red,-stealth] (a2) edge (a3);

      \path {
        (-0.5,-0.25) node (a1) {}
        ++(-4,-4) node (a2) {}
        ++(-2,2) node (a3) {}
      };
      \path[draw,red, -stealth] (a2) edge (a3);

      \path {
        (-1.7,-0.15) node (a1) {}
        ++(-3,-3) node (a2) {}
        ++(-1.8,1.8) node (a3) {}
      };
      \path[draw,red, -stealth] (a2) edge (a3);

      \path {
        (-1.3,-0.5) node (a1) {}
        ++(-4,-4) node (a2) {}
        ++(-1.2,1.2) node (a3) {}
      };
      \path[draw,red, -stealth] (a2) edge (a3);
    \end{scope}

    \path [->] (0,-0.5) edge node[label={[rotate=45]below: increasing $d$}]{} +(-5,-5);
    \path [->] (-5.25,-5.25) edge node[label={[rotate=-45]below: increasing $b$}]{} +(-5,+5);

  \end{tikzpicture}
\]

\paragraph{Coning} There exists
\[
  I_V \defeq \intv[\max_i(b'_i), \max_j(d'_j)]
\]
with nonzero morphisms to all $\intv[b'_j,d'_j]$ by using Lemma~\ref{lem:homdim} and the fact that $\max\limits_i(b'_i) \leq d_j'$ for all $j$ from Ineq.~\eqref{eq:b_condition}. We can thus construct the morphism
\[
  \begin{tikzcd}[ampersand replacement=\&, column sep=large]
    I_V \rar{\smat{1\\\vdots\\1}} \&
    \widebar{V}
  \end{tikzcd}
\]
with nonzero entries.

\paragraph{Constructed object and proof of indecomposability}
Our construction is summarized in the following picture:
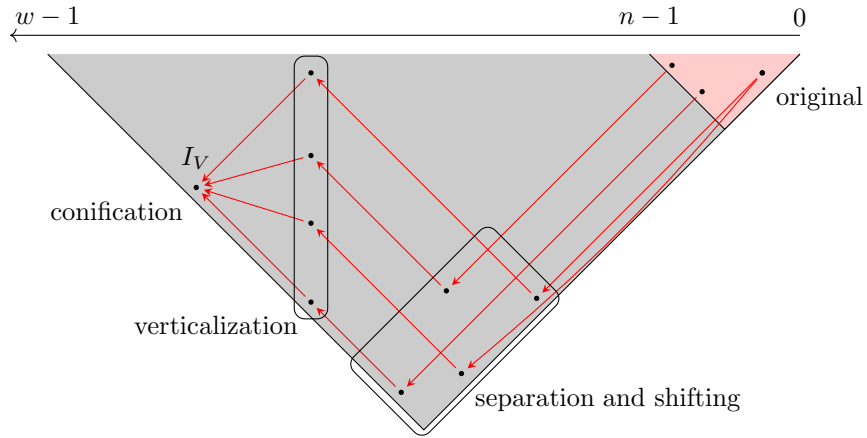
\[
  \begin{tikzpicture}
    \arsimple{5}{fill=black!20}
    \arsimple{1}{fill=red!20}

    \draw [->] (0,0.25)node[above]{$0$} -- (-2, 0.25)node[above]{$n-1$} -- (-10,0.25)node[above]{$w-1$} -- (-10.5,0.25);

    \begin{scope}[every node/.style={shape=circle,fill,minimum size=2pt,inner sep=0, outer sep=2pt}]
      \node (x) at (-8.025,-1.775) {};
      \path {
        (-0.5,-0.25) node (a1) {}
        ++(-3,-3) node (a2) {}
        ++(-3,+3) node (a3) {}
      };
      \path[draw,red,-stealth] (a1) edge (a2) (a2) edge (a3) (a3) edge (x);

      \path {
        (-0.5,-0.25) node (b1) {}
        ++(-4,-4) node (b2) {}
        ++(-2,2) node (b3) {}
      };
      \path[draw,red, -stealth] (b1) edge[bend left=5] (b2) (b2) edge (b3) (b3) edge (x);

      \path {
        (-1.7,-0.15) node (c1) {}
        ++(-3,-3) node (c2) {}
        ++(-1.8,1.8) node (c3) {}
      };
      \path[draw,red, -stealth] (c1) edge (c2) (c2) edge (c3) (c3) edge (x);

      \path {
        (-1.3,-0.5) node (d1) {}
        ++(-4,-4) node (d2) {}
        ++(-1.2,1.2) node (d3) {}
      };
      \path[draw,red, -stealth] (d1) edge (d2) (d2) edge (d3) (d3) edge (x);
    \end{scope}

    \node [label=above:$I_V$] at (x) {};
    \node[rounded corners, draw=black, fit=(a3) (b3) (c3) (d3)](FIt1) {};
    \node[rounded corners, inner sep=5pt, rotate fit=45, draw=black, fit=(a2) (b2) (c2) (d2)](FIt2) {};

    \node [anchor=north east, inner sep=5pt] at (x) {conification};
    \node [anchor=north east, inner sep=5pt] at (d3) {verticalization};
    \node [anchor=north west, inner sep=5pt] at (b2) {separation and shifting};
    \node [anchor=north west, inner sep=5pt] at (a1) {original};

    % \path [->] (0,-0.5) edge node[label={[rotate=45]below: increasing $d$}]{} +(-5,-5);
    % \path [->] (-5.25,-5.25) edge node[label={[rotate=-45]below: increasing $b$}]{} +(-5,+5);

  \end{tikzpicture}
\]
and we obtain $S\in \fun{\Af{4}}{\rep\Af{w}}$, a $\rep\Af{w}$-valued representation of
\[
  \Af{4}:
  \begin{tikzcd}[column sep=large]
    0 \rar & 1 \rar & 2 \rar & 3
  \end{tikzcd},
\]
determined by
\begin{equation}
  \label{eq:master_construction}
  S:
  \begin{tikzcd}[ampersand replacement=\&, column sep=large]
   	I_V \rar{\smat{1\\\vdots\\1}} \&
    \widebar{V} \rar{\smat{1 & & \\ & \ddots & \\ & & 1}} \&
    V' \rar{\smat{1 & & \\ & \ddots & \\ & & 1}} \&
    V
  \end{tikzcd}
\end{equation}
where $S(3)\defeq V$, $S(2) \defeq V'  = \bigoplus_{i=1}^m \intv[b_i,d'_i]$ is obtained by separate-and-shift, $S(1) \defeq \widebar{V} = \bigoplus_{i=1}^m \intv[b'_i,d'_i]$ by verticalization (of $V'$), and $S(0) \defeq I_V$ by coning.

Let us show that $S$ is indecomposable. Any $\phi \in \End(S)$ is given by a collection of maps $(\phi_0, \phi_1, \phi_2, \phi_3)$ making the diagram
\begin{equation}
\begin{tikzcd}
	I_V \rar \dar{\phi_0} & \widebar{V} \rar \dar{\phi_1} & V' \rar\dar{\phi_2} & V \dar{\phi_3}\\
    I_V \rar & \widebar{V} \rar & V' \rar & V
\end{tikzcd}
\label{eq:endo_comm}
\end{equation}
commutative.

Since the summands of $\widebar{V}$ are vertical, $\phi_1$ has matrix form
\begin{equation}
  \label{eq:diag}
 \left[
 \begin{array}{ccc}
 c_1 & & \\
 & \ddots & \\
 & & c_m
 \end{array}
 \right]
\end{equation}
for some scalars $c_1,c_2,\hdots, c_m \in K$ by Lemma~\ref{lem:vertical}.

By commutativity of the second and third boxes in Diagram~\eqref{eq:endo_comm}, the scalars in the matrix forms of $\phi_2$ and $\phi_3$ are nonzero only along the diagonal and with scalars along the diagonal given by $c_1,\hdots, c_m$. This involves some technical details, which we delay to Lemma~\ref{lem:prop}, where  Ineq.~\eqref{eq:b_condition} is essential.

By the commutativity of the first box in Diagram~\eqref{eq:endo_comm}, we get $c_1 = c_2 = \hdots = c_m = \phi_0 \in K$. Thus, $\End(S) \cong K$ is local and so $S$ is indecomposable.

We have shown that $S \in \fun{\Af{4}}{\rep{\Af{w}}}$ is indecomposable. By vertical stacking Lemma~\ref{lem:vertical_stacking}, this can be viewed as an indecomposable representation of the $4 \times w$ commutative grid $\Gf{4,w}$. This can be viewed (by padding with zeroes), as an indecomposable representation $S$ of $\Gf{}$ with finite support.

The original representation $V$ is obtained by line restriction of $S$ to the horizontal line $L:\mathbb{Z}\rightarrow\mathbb{Z}\times\mathbb{Z}$ that runs along the $4$th row.
%\qed
\end{proof}

Finally, let us prove the technical details used to prove the form of the endomorphisms of $S$.
\newpage
\begin{lemma}[Propagation of nonzeroes]
\label{lem:prop}
\leavevmode
\begin{enumerate}
\item
Let $\widebar{V} = \bigoplus\limits_{i=1}^m \intv[b'_i, d'_i]$ and
$V' = \bigoplus\limits_{i=1}^m \intv[b_i, d'_i]$ such that $b_i \leq b'_i \leq \min_j(d'_j)$ for all $i$, and let $\iota \defeq \smat{1 & & \\ & \ddots & \\ & & 1}:\widebar{V}\rightarrow V'$ be the morphism given in matrix form with only nonzeros in the diagonal given by the inclusions $f\itoi{b'_i,d'_i,b_i,d'_i} : \intv[b'_i, d'_i]\hookrightarrow \intv[b_i,d'_i]$, as above. Suppose that
\begin{equation}
  \label{lem:prop_comm}
  \begin{tikzcd}[ampersand replacement=\&, column sep=large]
    \widebar{V}
    \rar{\iota}
    \dar[swap]{\phi_1 = \smat{c_1 & & \\ & \ddots & \\ & & c_m}}  \&
    V'\dar{\phi_2 = \left[a_{j,i}f\itoi{b_i,d'_i,b_j,d'_j}\right]} \\
    \widebar{V}
    \rar[swap]{\iota} \&
    V'
  \end{tikzcd}
\end{equation}
is commutative. Then the matrix form of $\phi_2$ is $\smat{c_1 & & \\ & \ddots & \\ & & c_m}$.
\item
Let $V' = \bigoplus\limits_{i=1}^m \intv[b_i, d'_i]$ and $V = \bigoplus\limits_{i=1}^m \intv[b_i, d_i]$
such that $d'_i \geq d_i$ for all $i$, and let $\pi \defeq \smat{1 & & \\ & \ddots & \\ & & 1}:V'\rightarrow V$ be the morphism given in matrix form with only nonzeros in the diagonal given by the projections $f\itoi{b_i,d'_i,b_i,d_i} : \intv[b_i, d'_i]\twoheadrightarrow \intv[b_i,d_i]$, as above. Suppose that
\[
\begin{tikzcd}[ampersand replacement=\&, column sep=large]
V'
\rar{\pi}
\dar[swap]{\smat{c_1 & & \\ & \ddots & \\ & & c_m}}  \&
V\dar{\phi_3 = \left[a_{j,i}f\itoi{b_i,d_i,b_j,d_j}\right]} \\
V'
\rar[swap]{\pi} \&
V
\end{tikzcd}
\]
is commutative. Then the matrix form of $\phi_3$ is $\smat{c_1 & & \\ & \ddots & \\ & & c_m}$.
\end{enumerate}
\end{lemma}

Before giving the proof, we give an example where the conclusion of part \textit{1} fails to hold if the condition $b'_i \leq \min_j(d'_j)$ for all $i$ is not satisfied. If $\widebar{V}$ is vertical, then this condition is automatically satisfied, so we construct an example where this is not the case. We let
  $
    V' = \intv[1,2] \oplus \intv[1,1]
  $
  and
  $
    \widebar{V} = \intv[2,2] \oplus \intv[1,1].
  $
  The death-indices are fixed at $2$ and $ 1$, while $b_1 = 1 \leq 2 = b'_1$ and $b_2 = 1 \leq 1 = b'_2$. Note that $b'_1 = 2 \not\leq 1 = \min_j(d'_j)$. The commutative diagram~\eqref{lem:prop_comm} becomes
  \[
    \begin{tikzcd}[ampersand replacement=\&, column sep=huge, row sep=large]
      \intv[2,2] \oplus \intv[1,1]
      \rar{\iota\defeq\smat{1 f\itoi{2,2,1,2} & 0 \\ 0 & 1 f\itoi{1,1,1,1} }}
      \dar[swap]{\phi_1\defeq\smat{c_1\id_{\intv[2,2]} & 0 \\ 0 & c_2\id_{\intv[1,1]}}}  \&
      \intv[1,2] \oplus \intv[1,1]
      \dar{\phi_2 \defeq \smat{a_{11}f\itoi{1,2,1,2} & 0 \\ a_{21}f\itoi{1,2,1,1} & a_{22}f\itoi{1,1,1,1}}} \\
      \intv[2,2] \oplus \intv[1,1]
      \rar[swap]{\iota}  \&
      %\rar[swap]{\smat{1 f\itoi{2,2,1,2} & 0 \\ 0 & 1 f\itoi{1,1,1,1} }}  \&
      \intv[1,2] \oplus \intv[1,1]
    \end{tikzcd}.
  \]
  Note that $f\itoi{a,b,a,b} = \id_{\intv[a,b]}$ is the identity morphism.

  The composition $\iota \phi_1$ is
  \begin{equation}
    \label{eq:comm_right}
    \smat{1 f\itoi{2,2,1,2} & 0 \\ 0 & 1 f\itoi{1,1,1,1}}
    \smat{c_1\id_{\intv[2,2]} & 0 \\ 0 & c_2\id_{\intv[1,1]}}
    =
    \smat{c_1f\itoi{2,2,1,2} & 0 \\ 0 & c_2f\itoi{1,1,1,1}}.
  \end{equation}
  However, the composition $\phi_2 \iota$ is
  \begin{equation}
    \label{eq:comm_left}
    \smat{a_{11}f\itoi{1,2,1,2} & 0 \\ a_{21}f\itoi{1,2,1,1} & a_{22}f\itoi{1,1,1,1}}
    \smat{1 f\itoi{2,2,1,2} & 0 \\ 0 & 1 f\itoi{1,1,1,1}}
    =
    \smat{a_{11}f\itoi{2,2,1,2} & 0 \\ a_{21}f\itoi{1,2,1,1}f\itoi{2,2,1,2} & a_{22}f\itoi{1,1,1,1}}
    =
    \smat{a_{11}f\itoi{2,2,1,2} & 0 \\ 0 & a_{22}f\itoi{1,1,1,1}}
  \end{equation}
  where the lower-left entry is equal to $0$ since the composition $f\itoi{1,2,1,1}f\itoi{2,2,1,2} = 0$. This can be illustrated by the following.
  
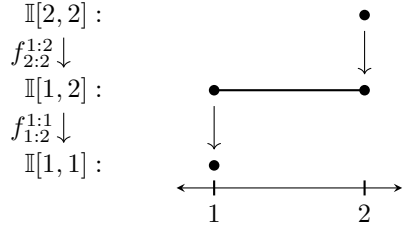
\[
    \begin{tikzpicture}
      \begin{scope}[every node/.style={shape=circle,fill,minimum size=4pt,inner sep=0, outer sep=4pt}]
        \node (a22) at (2,1) {};
        \draw[thick] (0,0)node(a12a){} -- (2,0)node(a12b){};
        \node (a11) at (0,-1) {};
      \end{scope}

      \draw[->] (a22) -- (a12b);
      \draw[->] (a12a) -- (a11);

      \node (a) at (-2,1) {$\intv[2,2]:$};
      \node (b) at (-2,0) {$\intv[1,2]:$};
      \node (c) at (-2,-1) {$\intv[1,1]:$};

      \draw[->] (a) --node[anchor=east]{$f\itoi{2,2,1,2}$} (b);
      \draw[->] (b) --node[anchor=east]{$f\itoi{1,2,1,1}$} (c);
      % \node[anchor=south east,yshift=5pt] at (a12b) {$f\itoi{2,2,1,2}$};
      % \node[anchor=north west,yshift=-5pt] at (a12a) {$f\itoi{1,2,1,1}$};
      \draw [stealth-stealth] (-0.5,-1.3) -- (2.5,-1.3);
      \draw[thick] (0,-1.2) -- (0,-1.4) node[anchor=north]{$\scriptsize{1}$};
      \draw[thick] (2,-1.2) -- (2,-1.4) node[anchor=north]{$\scriptsize{2}$};

    \end{tikzpicture}
  \]
  The birth index $2$ has exceeded the smallest death index $1$, and so the composition ``falls through'' and fails to be nonzero.

  The commutativity of the diagram implies the equality of Eq.~\eqref{eq:comm_right} and Eq.~\eqref{eq:comm_left}, from which we obtain $a_{11} = c_1$ and $a_{22} = c_2$, but $a_{21}$ is not determined. Thus, we cannot conclude that the matrix form of $\phi_2$ is the same diagonal.

\begin{proof}[Proof of Lemma~\ref{lem:prop}]
  The matrix formalism hides the morphisms $f\itoi{a,b,c,d}$, and the composition of two nonzero composable $f\itoi{a,b,c,d}$ can possibly be zero. What is important for this lemma is that the compositions should be nonzero where needed, so that the commutativity relation forces the scalars $a_{j,i}$ to be zero off-diagonal and the same as $c_j$ along the diagonal.

\begin{enumerate}
\item
The composition $\phi_2 \iota$ is equal to
\[
	\left[ (a_{j,i} f\itoi{b_i,d'_i,b_j,d'_j}) ( 1 f\itoi{b'_i,d'_i,b_i,d'_i}) \right].
\]
The entries $(j,i)$ with $f\itoi{b_i,d'_i,b_j,d'_j}$ nonzero are given by
$
	b_j \leq b_i \leq d'_j \leq d'_i
$
by Lemma~\ref{lem:homdim} (these are the entries where $a_{j,i}$ are present). For each such entry, the corresponding morphism in the composition $\phi_2\iota$ is given by
\[
	f\itoi{b_i,d'_i,b_j,d'_j}  f\itoi{b'_i,d'_i,b_i,d'_i}:\intv[b'_i,d'_i] \rightarrow \intv[b_i,d'_i] \xrightarrow{f\itoi{b_i,d'_i,b_j,d'_j} \neq 0}\intv[b_j,d'_j].
\]
Using Lemma~\ref{lem:homdim} and the fact that $b'_i \leq \min_j(d'_j) \leq d'_j$, we see that this composition is nonzero.

Thus, $\phi_2\iota$ has scalars in its matrix form the same as $\phi_2$ (nonzeros cannot go to zeros). Equating with the other composition $\iota\phi_1$ in the commutative diagram, $a_{j,j} = c_j$ for all $j$ and $a_{j,i} = 0$ for $i\neq j$, as required.

\item The composition $\phi_3 \pi$ is equal to
\[
	\left[ (a_{j,i} f\itoi{b_i,d_i,b_j,d_j}) ( 1 f\itoi{b_i,d'_i,b_i,d_i}) \right].
\]
The entries $(j,i)$ with $f\itoi{b_i,d_i,b_j,d_j}$ nonzero are given by
$
	b_j \leq b_i \leq d_j \leq d_i
$
by Lemma~\ref{lem:homdim}. In these entries, the composition
\[
	\intv[b_i,d'_i] \rightarrow \intv[b_i,d_i] \xrightarrow{f\itoi{b_i,d_i,b_j,d_j} \neq 0}\intv[b_j,d_j]
\]
is also nonzero by Lemma~\ref{lem:homdim} and the fact that $d_i \leq d'_j$. Then the conclusion is obtained by a similar argument as above.
%\qed
\end{enumerate}
\end{proof}

\paragraph{Dual Construction} We note that a dual construction is possible. Instead of constructing to the left, we can proceed to the right and obtain
\[
\begin{tikzcd}[ampersand replacement=\&, column sep=large]
	V \rar{\smat{1 & & \\ & \ddots & \\ & & 1}} \&
    V'' \rar{\smat{1 & & \\ & \ddots & \\ & & 1}} \&
    \widebar{\widebar{V}} \rar{\smat{1 & \hdots & 1}} \&
    I'_V
\end{tikzcd}
\]
where $V''  = \bigoplus_{i=1}^m \intv[b'_i,d_i]$ is obtained by separate-and-shift choosing smaller (possibly negative) birth-indices that are distinct, and $\widebar{\widebar{V}} = \bigoplus_{i=1}^m \intv[b'_i,d'_i]$ by verticalization, and $I'_V$ by coning. Note that in this case, the correct condition to use for separate-and-shift is to have $b'_i$ all distinct with $b'_i \leq b_i$ for all $i$ and $b_j' \leq \frac{b_i'+d_i}{2}$ for all pairs $i,j$.
Then we use $\mu\defeq\min\left(\frac{b_i'+d_i}{2}\right)$ and $d_i' \defeq 2\mu-b_i'$ for verticalization. Coning is given by $I'_V \defeq \intv[\min{b_i'},\min{d'_j}]$. Here, we can check that $b_j' \leq d_i' \leq d_i$ for all pairs $i,j$, which allows the $1$'s in the matrix form of the morphism to $I'_V$, and enables the dual version of propagation of nonzeros.

\begin{remark}
  \label{rem:shape}
  Note that the primal and dual constructions can be combined into one, yielding the following object
  \[
  	\begin{tikzcd}
      I_V \rar & \widebar{V} \rar & V' \rar & V \rar & V'' \rar & \widebar{\widebar{V}} \rar & I'_V
    \end{tikzcd}
  \]
  of $\fun{\Af{7}}{\rep\Af{w}}$, where $V'$ and $V''$ are obtained from separate-and-shift to the right and left, $\widebar{V}$ and $\widebar{\widebar{V}}$ from verticalization, and $I_V$ and $I'_V$ from coning. This corresponds to a $2$D persistence module by
  Stacking Lemma~\ref{lem:vertical_stacking}.

  We note that the support of the obtained module is contained in a rectangle, with the lower right and upper left corners both supporting one-dimensional vector spaces. Moreover its endomorphism ring is isomorphic to $K$.
  In general, we call $2$D persistence modules with these properties \emph{candy modules} as they are wrapped like candies with two handles we can manipulate.
  These modules will play a central role in the construction of section~\ref{sec:lotr}.
\end{remark}
%%% Local Variables:
%%% mode: latex
%%% TeX-master: "arxiv_main"
%%% End:

\section{Consequences on the homology of supports}\label{sec:zoology}

One consequence of our construction is the easy construction of examples of indecomposable representations of $\Gf{}$ (indecomposable $2$D persistence modules) whose support has a ``hole''. We contrast this with interval representations, a restricted class of representations that has been studied before \cite{bjerkevik2016stability,dey2018computing}. By definition, an \emph{interval representation} associates to each vertex a vector space with dimension at most $1$, its internal linear maps are always identity between non-zero vector spaces, and its support satisfies a convexity and connectedness condition. In particular, the support of any interval representation of $\Gf{m,n}$ always has a distinctive staircase shape~\cite{thin-decomposability}.

\subsection{Indecomposables with holes in support}
Note that while the technical details of our construction use nonnegative indices, it still works after shifting all indices. Let $V=\intv[-1,-1]\bigoplus\intv[1,1]$, whose form is chosen for symmetry.
We follow Remark~\ref{rem:shape} and obtain by the primal and dual construction:
\[
  H:
	\begin{tikzcd}
    I_V \rar & \widebar{V} \rar & V' \rar & V \rar & V'' \rar & \widebar{\widebar{V}} \rar & I'_V
  \end{tikzcd}
\]
% where $V'$ and $V''$ are obtained from separate-and-shift to the right and left, $\widebar{V}$ and $\widebar{\widebar{V}}$ from verticalization, and $I_V$ and $I'_V$ from coning.
% and use the Stacking Lemma~\ref{lem:vertical_stacking} % to turn the representation-valued representation $H$ into
% obtain a $2$D persistence module,
% as
corresponding to the $2$D persistence module
\begin{equation}
  \label{eq:beast1}
  \begin{tikzcd}[ampersand replacement=\&, column sep=1em]
    K \rar \& K \rar \& K \\
    K \rar \uar \& K \rar{\smat{1\\0}} \uar \& K^2 \rar{\smat{1 & 0}} \uar{\smat{1 & 1}} \& K \rar \& K\\
    K \rar \uar \& K \rar{\smat{1\\0}} \uar \& K^2 \rar \uar \& K^2 \rar \uar{\smat{1 & 0}} \& K^2 \rar{\smat{0 & 1}} \uar{\smat{1 & 0}} \& K \rar \& K \\
    \& \& \& \& K \rar \uar{\smat{1 \\ 0}} \& \textcolor{red}{0} \rar \uar \& K \uar \\
    \& \& \& \& K \rar\uar \& K \rar{\smat{1\\0}}\uar \& K^2 \rar\uar{\smat{0 & 1}} \& K^2 \rar \& K^2 \rar{\smat{0 & 1}} \& K \rar \& K\\
    \& \& \& \& \& \& K \rar\uar{\smat{0\\1}} \& K \rar{\smat{0\\1}}\uar{\smat{0\\1}} \& K^2 \rar{\smat{0 & 1}}\uar \& K \rar\uar \& K\uar \\
    \& \& \& \& \& \& \& \& K \rar \uar{\smat{1\\1}} \& K\rar\uar \& K\uar
  \end{tikzcd}
\end{equation}
where all but one zero vector space outside the support  are omitted for the sake of clarity. Maps between identical vector spaces are to be understood as the identity maps.
Then by the arguments in the proof for Theorem~\ref{thm:main} and the dual construction, we see that $\End(H)\cong K$, and thus $H$ is indecomposable. Moreover its support has a hole as the central vector space is $0$.

One way to formally define the notion of the hole in the support is to consider the cubical complex induced by the support of a $2$D persistence module $V$  as follows. Recall that the support of $V$ is the set of points $\supp(V) = \{x \in \Gf{} \suchthat V(x) \neq 0\}$. We then build a cubical complex by adding an edge between $x$ and $y$ whenever $x,y\in\supp(V)$ and $x$ and $y$ are adjacent on the $2$D grid. Whenever we have four points $(i,j),(i,j+1),(i+1,j+1),(i+1,j) \in \supp(V)$ for some $i,j\in\mathbb{Z}$, we add a unit square with those corner vertices. This is the clique-cubical complex of $\supp(V)$.
We then consider holes in the support as being non-trivial classes in the $1$-dimensional homology of this cubical complex. We say that the support of a $2$D persistence module has $n$ holes if the $1$-dimensional homology of the cubical complex built on the support has $n$ linearly independent classes.
It is clear that the representation given in Diagram~\eqref{eq:beast1} has a hole in this sense.

Our construction can then be used to generate indecomposable $2$D persistence modules having an arbitrary finite number of holes in its support.
\begin{corollary}\label{cor:holes}
  Let $n \in \{0,1,2,\hdots\}$. There exists an indecomposable $2$D persistence module whose support has exactly $n$ holes.
\end{corollary}

\begin{proof}
  Let $n$ be as given, and fix $V=\bigoplus_{i=0}^n \intv[2i,2i]$.

  Note that $\supp V$ has exactly $n+1$ connected components.
  Again we define the indecomposable representation
  \[
    H:
    \begin{tikzcd}
      I_V \rar & \widebar{V} \rar & V' \rar & V \rar & V'' \rar & \widebar{\widebar{V}} \rar & I'_V
    \end{tikzcd}
  \]
  by using the construction in Theorem~\ref{thm:main} and its dual. By the Stacking Lemma~\ref{lem:vertical_stacking}, $H$ is viewed as a $2$D persistence module.

  By construction, each representation $X \in \{V',V'',\widebar{V},\widebar{\widebar{V}},I_V,I'_V\}$ has connected  support because all of the intervals composing $X$ have a common index (depending on $X$) where they are non-zero. Moreover, the right extremities of $V'$, $\widebar{V}$ and $I_V$ are aligned, and thus the cubical complex of the support of $I_V\rightarrow \widebar{V}\rightarrow V'$ is contractible.
  Similarly the left extremities of $V''$, $\widebar{\widebar{V}}$ and $I'_V$ are also aligned, and the cubical complex of the support of $V''\rightarrow\widebar{\widebar{V}}\rightarrow I'_V$ is also contractible.

  Finally, the support of $V$ is included in both the support of $V'$ and $V''$, ensuring that every neighboring pair of connected components in the support of $V$ can be used to generate a non-trivial $1$-homology class in the cubical complex of $\supp(H)$. There are $n$ such pairs, and it is clear that $H$ has exactly $n$ holes in its support.

  Below, we illustrate the clique cubical complex of the support of $H$.
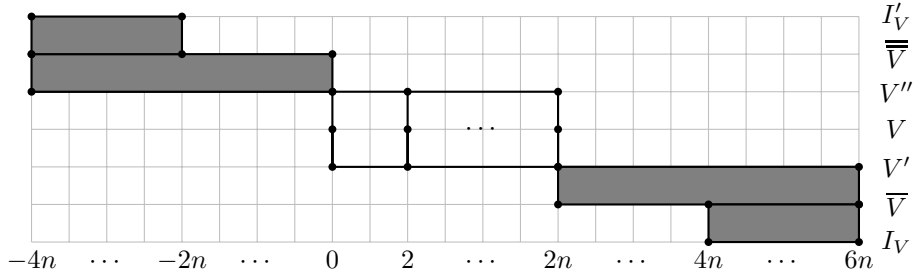
  \[
  \begin{tikzpicture}[scale=0.5]
    \tikzset{dotstyle/.style={circle,fill=black,inner sep=0pt,minimum size=3pt}}
    \draw[gray!50] (-8,0) grid (14,6);

    \draw[thick] (0,3) node[dotstyle]{} -- +(0,-1) node[dotstyle]{} -- +(0,1) node[dotstyle]{};
    \draw[thick] (2,3) node[dotstyle]{} -- +(0,-1) node[dotstyle]{} -- +(0,1) node[dotstyle]{};
    \draw[thick] (6,3) node[dotstyle]{} -- +(0,-1) node[dotstyle]{} -- +(0,1) node[dotstyle]{};
    \node at (4,3){$\cdots$};

    % separate and shift
    \draw[thick] (0,2) -- (14,2);
    \draw[thick] (-8,4) -- (6,4);

    % verticalize
    \draw[thick, fill=gray] (6,1)node[dotstyle]{} -- (14,1) node[dotstyle]{} -- (14,2)node[dotstyle]{} -- (6,2)node[dotstyle]{} -- cycle;

    \draw[thick, fill=gray] (0,4)node[dotstyle]{} -- (-8,4) node[dotstyle]{} -- (-8,5)node[dotstyle]{} -- (0,5)node[dotstyle]{} -- cycle;

    % coning
    \draw[thick, fill=gray] (10,0)node[dotstyle]{} -- (14,0) node[dotstyle]{} -- (14,1)node[dotstyle]{} -- (10,1)node[dotstyle]{} -- cycle;

    \draw[thick, fill=gray] (-4,5)node[dotstyle]{} -- (-8,5) node[dotstyle]{} -- (-8,6)node[dotstyle]{} -- (-4,6)node[dotstyle]{} -- cycle;

    \node[anchor=north] at (-8,0){$-4n$};
    \node[anchor=north] at (-6,0){\vphantom{$2n$}$\cdots$};
    \node[anchor=north] at (-4,0){$-2n$};
    \node[anchor=north] at (-2,0){\vphantom{$2n$}$\cdots$};
    \node[anchor=north] at (0,0){$0$};
    \node[anchor=north] at (2,0){$2$};
    \node[anchor=north] at (4,0){\vphantom{$2n$}$\cdots$};
    \node[anchor=north] at (6,0){$2n$};
    \node[anchor=north] at (8,0){\vphantom{$2n$}$\cdots$};
    \node[anchor=north] at (10,0){$4n$};
    \node[anchor=north] at (12,0){\vphantom{$2n$}$\cdots$};
    \node[anchor=north] at (14,0){$6n$};

    \foreach \L [count=\h from 0] in {I_V,\widebar{V},V',V,V'',\widebar{\widebar{V}},I'_V}{
      \node at (15,\h){$\L$};
    }
  \end{tikzpicture}
\]
Note that this realizes a topological suspension of the support of $V$.
%\qed
\end{proof}

\subsection{A minimal construction}
The construction used to prove Theorem~\ref{thm:main} works for general $1$D persistence modules $V$, but does not consider any size minimality for particular cases of $V$. For example, the persistence module in Diagram~\eqref{eq:beast1} is not the smallest example of an indecomposable $2$D persistence module whose support has a hole. Indeed, using only the primal construction to obtain the bottom half of Diagram~\eqref{eq:beast1}, and ``capping off'' the top of $V$ with an interval would have yielded an example with smaller support.

Before discussing a minimal example in Theorem~\ref{thm:smallest}, let us study a local construction that we use repeatedly for its proof. Consider the family ${\cal P}$ of $2$D persistence modules whose support is contained in a small ``cross'' shape. The small cross shape is the $5$ vertices $a$, $b$, $c$, $d$, $e$ (filled-in circles) and filled-in arrows $\alpha$, $\beta$, $\gamma$, $\delta$:
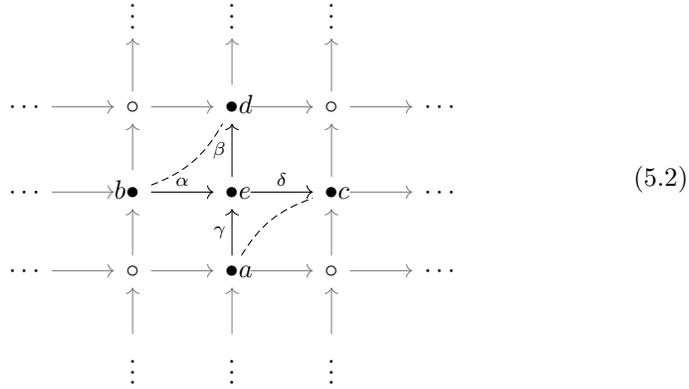
\begin{equation}
  \label{eq:cross_string_alg}
  \begin{tikzcd}
    & \vdots & \vdots & \vdots &  \\
    \cdots \arrow[gray,r] & \circ \arrow[gray,r] \arrow[gray,u] & \bullet \mathrlap{d} \arrow[gray,r] \arrow[gray,u] & \circ \arrow[gray,r] \arrow[gray,u] & \cdots \\
    \cdots \arrow[gray,r] & \mathllap{b} \bullet \arrow[gray,u] \arrow[r,"\alpha"] \ar[ur, no head, dashed, bend right=20]{} & \bullet\mathrlap{e} \arrow[r,"\delta"] \arrow[u,"\beta"] & \bullet\mathrlap{c} \arrow[gray,u] \arrow[gray,r] & \cdots \\
    \cdots \arrow[gray,r] & \circ \arrow[gray,u] \arrow[gray,r] & \bullet \mathrlap{a} \arrow[u,"\gamma"] \arrow[gray,r] \ar[ur, no head, dashed, bend left=20]{} & \circ \arrow[gray,u] \arrow[gray,r] & \cdots \\
    & \vdots \arrow[gray,u] & \vdots \arrow[gray,u] & \vdots \arrow[gray,u] &
  \end{tikzcd}
\end{equation}
where the unfilled circles and grayed-out arrows are not part of the support.

Consider $M \in \mathcal{P}$. Since the upper-left and lower-right vertex are not part of the support, $M$ associates the zero vector space to those vertices. By commutativity relations, the compositions $M(\beta)M(\alpha)$ and $M(\delta)M(\gamma)$ must be zero. These zero relations are denoted by the dashed lines. Thus, we can also treat $M\in \mathcal{P}$ as a representation of the corresponding bound quiver.

\begin{restatable}{lemma}{crosslem}
  \label{lem:cross}
	There is no indecomposable $T \in \cal P$ whose support contains both the bottom and right vertices ($a$ and $c$) or both the left and top vertices ($b$ and $d$).
\end{restatable}
We provide a proof in the Appendix.

\begin{remark}
  Note that the quiver formed by the filled-in vertices and arrows in Diagram~\eqref{eq:cross_string_alg}, together with the relations $\delta\gamma=0$ and $\beta\alpha=0$, generates what is known as a string algebra. Lemma~\ref{lem:cross} in fact follows immediately from a known characterization of indecomposable representations of string algebras as the string modules and band modules \cite{butler1987auslander}.
\end{remark}

The important consequence of Lemma~\ref{lem:cross} is that given any $2$D persistence module $M$ on such a cross shape with zero relations as in Diagram~\eqref{eq:cross_string_alg}, bases for its vector spaces can be chosen that allows for a decomposition $M = M_1 \oplus M_2$ into the direct sum of two representations such that $M_1(a) = 0$ and $M_2(c) = 0$, or symmetrically $M=M'_1 \oplus M'_2$ with $M'_1(b) = 0$ and $M'_2(d) = 0$.

\begin{theorem}
  \label{thm:smallest}
  The minimal number of vertices in the support of an indecomposable $2$D persistence module whose support has a hole is equal to $11$. For example, the $2$D persistence module
  \[
    M:
    \begin{tikzcd}[ampersand replacement=\&]
      K \rar{1} \& K \rar \& K \\
      K \rar \uar{0} \& \textcolor{red}{0} \rar \uar \& K \uar \\
      K \rar \uar \& K \rar{\smat{1\\0}}\uar \& K^2 \rar{\smat{1 & 0}}\uar{\smat{0 & 1}} \& K\\
      \& \& K \rar \uar{\smat{1\\1}} \& K \uar
    \end{tikzcd}
  \]
  realizes this minimum.
\end{theorem}

\begin{proof}
	The support of $M$ contains exactly $11$ vertices and has a hole.
	An elementary computation shows that the endomorphism ring of $M$ is isomorphic to $K$, and thus $M$ is indecomposable.
	It is then sufficient to prove that there is no indecomposable $2$D persistence module whose support has a hole and contains $10$ or less vertices. The existence of a hole requires at least $8$ vertices in the support.

	\paragraph{$8$ vertices} Suppose that $T$ is an indecomposable $2$D persistence module whose support has exactly $8$ vertices and has a hole. Then $T$ is of the form:
	\[
    T:
    \begin{tikzcd}[ampersand replacement=\&]
      X_2 \rar{f_3} \& X_3\rar{f_4} \& B \\
      X_1 \rar{} \uar{f_2} \& 0 \uar \rar \& Y_3\uar{g_4} \\
      A \uar{f_1}\rar{g_1} \& Y_1 \uar \rar{g_2} \& Y_2\uar{g_3}
    \end{tikzcd}
	\]
	By commutativity, both $f_3f_2$ and $g_3g_2$ are zero maps.
	We can thus construct the following:
	\[
    T_A:
    \begin{tikzcd}[ampersand replacement=\&,column sep=small]
			\Ker(f_3) \rar \& 0\rar \& 0 \\
			X_1 \rar{} \uar{f_2} \& 0 \uar \rar \& 0\uar \\
			A \uar{f_1}\rar{g_1} \& Y_1 \uar \rar{g_2} \& \Ker(g_3)\uar
    \end{tikzcd}
    \hspace{1em}
    T_B:
    \begin{tikzcd}[ampersand replacement=\&,column sep=small]
			\Coim(f_3) \rar{\hat{f_3}} \& X_3\rar{f_4} \& B \\
			0 \rar \uar \& 0 \uar \rar \& Y_3\uar{g_4} \\
			0 \uar\rar \& 0 \uar \rar \& \Coim(g_3)\uar{\hat{g_3}}
    \end{tikzcd}
	\]
  where $\hat{f_3}$ and $\hat{g_3}$ are the induced maps on the respective coimages.
	It is a simple matter to check that $T\cong T_A\oplus T_B$. Since the support of $T$ has exactly $8$ vertices, $A$ and $B$ are nonzero, and so $T_A$ and $T_B$ are nonzero. This contradicts the fact that $T$ is indecomposable.

	In order to prepare for the remaining cases, let us rephrase this argument using Lemma~\ref{lem:cross} and its proof. First, we view $T$ as a representation of
  \begin{equation}
    \label{eq:eighthole}
    \begin{tikzcd}[ampersand replacement=\&]
      x_2 \rar \& x_3\rar \& b \\
      x_1 \uar \ar[ur, no head, dashed, bend left=20]{} \&    \& y_3\uar \\
      a \uar\rar \& y_1 \rar \ar[ur, no head, dashed, bend right=20]{} \& y_2\uar
    \end{tikzcd}.
  \end{equation}
  We then restrict our attention to the ``partial crosses'' centered at $x_2$ and $y_2$, and use a similar argument as in the proof of Lemma~\ref{lem:cross}. By choices of bases, we obtain summands where the vector spaces at vertices $x_1$ and $x_3$ are not both nonzero, and where the vector spaces at vertices $y_1$ and $y_3$ are not both nonzero.

  These choices can be done simultaneously for the two partial crosses, as they share no vertices.
  In general, such a ``local'' decomposition does not induce a decomposition of a representation as a whole.
  % However, in our case, there is no path from $a$ to $b$ that does not go through one of the partial cross.

  However, in this case, we can build the following two subrepresentations $T_A$ and $T_B$ of $T$.
  For $T_A$, we take the decomposition summands over the two crosses which are non-zero at vertices $x_1$ and $y_1$, and then add the vector space $A$ at vertex $a$ with the linear maps from $A$ to $X_1$ and $A$ to $Y_1$. Using Lemma~\ref{lem:cross} ensures that we can simply fix the vector spaces of $T_A$ at $x_3$, $y_3$ and $b$ to be $0$, and that the result is indeed a subrepresentation. Note that we obtain the same $T_A$ as defined previously.
  Symmetrically, we obtain $T_B$, and it is a simple matter to show that $T \cong T_A\oplus T_B$.

	\paragraph{$9$ vertices} A support with $9$ vertices and containing a hole can only be obtained by adding one vertex somewhere on the outside of Diagram~\eqref{eq:eighthole}. Depending on where these vertices are added, an appropriate zero relation needs to be imposed in order to ensure the embeddability in the equioriented commutative $2$D grid. For example, with $z$ the additional vertex in the support, the underlying bound quiver becomes
  
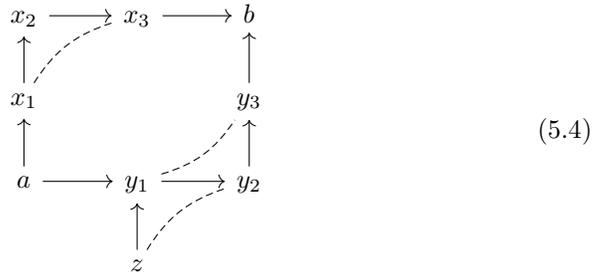
\begin{equation}
    \begin{tikzcd}[ampersand replacement=\&]
      x_2 \rar \& x_3\rar \& b  \\
      x_1 \uar \ar[ur, no head, dashed, bend left=20]{} \&    \& y_3\uar \\
      a \uar\rar \& y_1 \rar \ar[ur, no head, dashed, bend right=20]{} \& y_2 \uar \\
      \& z \uar \ar[ur, no head, dashed, bend left=20]{} \&
    \end{tikzcd}
    \label{diag:nine_vertices}
  \end{equation}
  where we still have partial crosses centered at $x_2$ and $y_2$. The same argument as the one in the case of $8$ vertices, using these two partial crosses, shows that a $2$D persistence module with such a support cannot be indecomposable.

  Note that the extra vector space at vertex $z$ is handled in two different ways depending on the position of $z$. If $z$ is part of one of the partial crosses at $x_2$ or $y_2$ then it is automatically split in a direct sum of vector spaces, one being part of $T_A$ and the other part of $T_B$ by the use of Lemma~\ref{lem:cross}.

  If $z$ is not part of one of the crosses (for example Diagram~\eqref{diag:nine_vertices}), then the vector space on its adjacent vertex ($y_1$ in Diagram~\eqref{diag:nine_vertices}) is split into a part in $T_A$ and a part in $T_B$. By construction, at least one of those must be $0$. The vector space at vertex $z$ together with the map connecting it to the rest of the diagram can be added to the subrepresentation with the not necessarily zero part. For example, in Diagram~\eqref{diag:nine_vertices}, the subrepresentation $T_B$ is $0$ at vertex $y_1$. The vector space and linear map of $T$ associated to $z$ can thus be added to the $T_A$ obtained in the $8$ vertices case while guaranteeing that $T_A$ is still a subrepresentation.

	\paragraph{$10$ vertices} The case with $10$ vertices can be obtained in two ways.

	First, the hole in Diagram~\eqref{eq:eighthole} can be made wider by adding two vertices (and corresponding arrows) to the support, for example between $x_3$ and $b$ and between $y_1$ and $y_2$. This means that the $2$D persistence module will need to encompass two $0$ vector spaces in its center. The proof for this case is similar to the one for $8$ vertices.

	The second case is obtained when we add two vertices (and corresponding arrows) to Diagram~\eqref{eq:eighthole} on the outside.
  Depending on where these vertices are added, zero relations or commutativity relations need to be imposed in order to ensure the embeddability in the equioriented commutative $2$D grid, as we shall see below.

  Thus we further subdivide this case. First, in the case that the two new vertices are not adjacent, we only need to deal with new zero relations. For example, both vertices can be attached at corner $y_2$, as below:
  
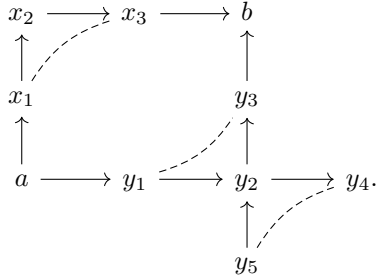
\[
    \begin{tikzcd}[ampersand replacement=\&]
      x_2 \rar \& x_3\rar \& b \& \\
      x_1 \uar \ar[ur, no head, dashed, bend left=20]{} \&    \& y_3\uar \& \\
      a \uar\rar \& y_1 \rar \ar[ur, no head, dashed, bend right=20]{} \& y_2 \uar \rar \& y_4. \\
      \& \& y_5\uar\ar[ur, no head, dashed, bend left=20]{} \&
    \end{tikzcd}
  \]
  The argument using the partial crosses can be used as before.

  A second possibility is that the two new vertices are adjacent, for example:
  
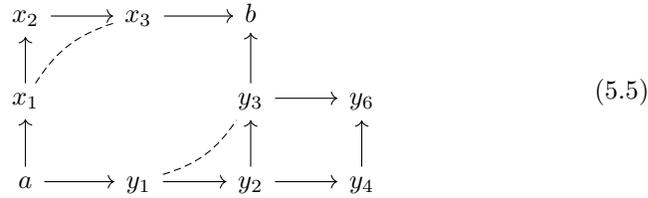
\begin{equation}
    \label{diag:hole_comm_rel}
    \begin{tikzcd}[ampersand replacement=\&]
      x_2 \rar \& x_3\rar \& b \& \\
      x_1 \uar \ar[ur, no head, dashed, bend left=20]{} \&    \& y_3\uar \rar \& y_6 \\
      a \uar\rar \& y_1 \rar \ar[ur, no head, dashed, bend right=20]{} \& y_2 \uar \rar \& y_4\uar
    \end{tikzcd}
  \end{equation}
  where we need to impose the commutative relation in the right square.

  Even in this case, a variant of the argument in Lemma~\ref{lem:cross} is still feasible. Suppose a $2$D persistence module $T$ has Diagram~\eqref{diag:hole_comm_rel} as its support. We construct subrepresentations $T_A$ and $T_B$. For $T_A$, its vector spaces at vertices $x_2$, $x_1$, $a$, and $y_1$ are obtained as before, using the same partial cross argument around vertex $x_2$.

  We set the vector space $T_A(y_2)$ of $T_A$ at vertex $y_2$ as the image of $T_A(y_1) \defeq T(y_1)$ under the linear map of $T$ from $y_1$ to $y_2$.  Similarly, the vector space $T_A(y_4)$ is the image of $T(y_1)$ under the composition of maps from $y_1$ to $y_2$ to $y_4$.

  By commutativity and the zero relation, $T_A(y_4)$ is in the kernel of the map from $y_4$ to $y_6$.
  Thus, $T_A(y_3)$ and $T_A(y_6)$ can be set to $0$, along with $T_A(x_3)$ and $T_A(b)$ as before, while guaranteeing that $T_A$ is indeed a subrepresentation of $T$.

  Similarly, a subrepresentation $T_B$ that is $0$ on $a$, $x_1$, $y_1$ can be constructed such that $T\cong T_A\oplus T_B$. As before, we note that $T_A$ is nonzero at $a$ and $T_B$ is nonzero at $b$, showing that $T$ is in fact decomposable.

  Finally, it is possible to have adjacent new vertices without inducing a commutative relation either by having an arrow elongating a path such as:
  
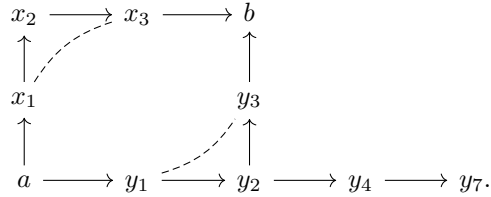
\[
    \begin{tikzcd}[ampersand replacement=\&]
      x_2 \rar \& x_3\rar \& b \& \& \\
      x_1 \uar \ar[ur, no head, dashed, bend left=20]{} \&  \& y_3\uar \& \\
      a \uar\rar \& y_1 \rar \ar[ur, no head, dashed, bend right=20]{} \& y_2 \uar \rar \& y_4\rar \& y_7.
    \end{tikzcd}
  \]
  where the arguments from the 9 vertex case apply.
  Otherwise the new arrow points in the opposite direction of the existing one such as:
  
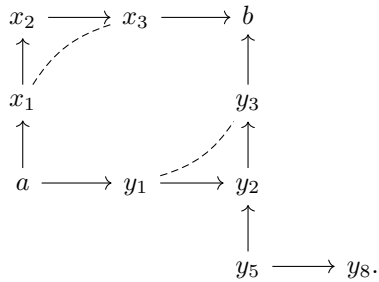
\[
    \begin{tikzcd}[ampersand replacement=\&]
      x_2 \rar \& x_3\rar \& b \& \\
      x_1 \uar \ar[ur, no head, dashed, bend left=20]{} \&  \& y_3\uar \& \\
      a \uar\rar \& y_1 \rar \ar[ur, no head, dashed, bend right=20]{} \& y_2 \uar
      \\
      \& \& y_5\rar \uar \& y_8.
    \end{tikzcd}
  \]
  This case is a bit more involved.
  However a careful computation on the subpart induced by vertices $\{y_1,y_2,y_3,y_5,y_8\}$ leads to a decomposition similar to the one of Lemma~\ref{lem:cross}, enabling the same arguments to follow as in the other cases.

  In all cases considered, a persistence module with support of the prescribed shape and number of vertices cannot be indecomposable. This shows the claimed result.
%\qed
  \end{proof}

%%% Local Variables:
%%% mode: latex
%%% TeX-master: "arxiv_main"
%%% End:

\section{A representation to contain them all}\label{sec:lotr}

In Section~\ref{sec:results}, we showed that for every $1$D persistence module in $\rep\Afn$, we can build an indecomposable $2$D persistence module containing it as a line restriction.
We can go further and build an indecomposable $2$D persistence module with infinite support that contains all $1$D persistence modules with finite support as line restrictions.

Recall that the primal and dual construction for Theorem~\ref{thm:main} starting from a $1$D persistence module $V$ results in a $2$D candy module (see Remark~\ref{rem:shape}) of the form
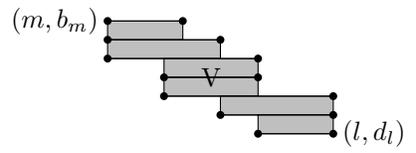
\[
  \newcommand{\drawprofile}[8]{
    % d1, d2, bV, dV, d', b1, b2, label
    % intervals are:
    % (0,d1), (0,d2), (0,dV)    [dual]
    % (bV,dV)                   [V]
    % (bV,d'), (b1,d'), (b2,d') [primal]
    \def\hh{0.5}
    \foreach \b/\d [count=\level from 0] in {0/#1, 0/#2, #3/#4}{
      \pgfmathsetmacro{\z}{\hh * -\level}
      \pgfmathsetmacro{\zz}{\hh * -(\level+1)}
      \draw[fill=gray!50] (\b,\z) rectangle (\d,\zz);
    }
    \foreach \b/\d [count=\level from 4] in {#3/#4, #6/#5, #7/#5} {
      \pgfmathsetmacro{\z}{\hh * -\level}
      \pgfmathsetmacro{\zz}{\hh * -(\level-1)}
      \draw[fill=gray!50] (\b,\z) rectangle (\d,\zz);
    }
    \foreach \b/\d [count=\level from 0] in {0/#1, 0/#2, 0/#4, #3/#4, #3/#5, #6/#5, #7/#5} {
      \pgfmathsetmacro{\z}{\hh * -\level}
      \draw (\b,\z)node[dotstyle]{} --node{{\ifthenelse{\level=3}{#8}{}}} (\d,\z)node[dotstyle]{};
    }
    \node[anchor=east] at (0,0) {$(m,b_m)$};
    \node[anchor=west] at (#5,-3) {$(l,d_l)$};
  }
  \begin{tikzpicture}[scale=0.5]
    \tikzset{dotstyle/.style={circle,fill=black,inner sep=0pt,minimum size=3pt}}
    \drawprofile{2}{3}{1.5}{4}{6}{3}{4}{V}
  \end{tikzpicture}.
\]

Given two candy modules $A$ and $B$, we define the following concatenation operation.
Let $r_A$ be the rightmost vertex in the lowest interval of $A$.
We translate $B$ such that the leftmost vertex of the top interval of $B$ (denoted $l_B$) is located below and to the right of $r_A$:
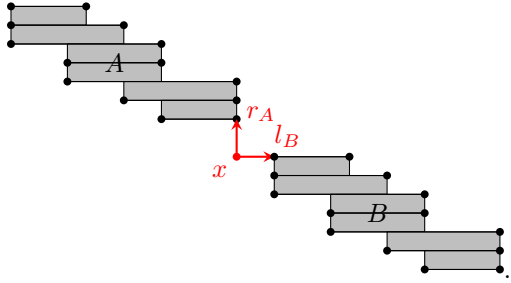
\[
  \newcommand{\drawprofile}[8]{
    % d1, d2, bV, dV, d', b1, b2, label
    % intervals are:
    % (0,d1), (0,d2), (0,dV)    [dual]
    % (bV,dV)                   [V]
    % (bV,d'), (b1,d'), (b2,d') [primal]
    \def\hh{0.5}
    \foreach \b/\d [count=\level from 0] in {0/#1, 0/#2, #3/#4}{
      \pgfmathsetmacro{\z}{\hh * -\level}
      \pgfmathsetmacro{\zz}{\hh * -(\level+1)}
      \draw[fill=gray!50] (\b,\z) rectangle (\d,\zz);
    }
    \foreach \b/\d [count=\level from 4] in {#3/#4, #6/#5, #7/#5} {
      \pgfmathsetmacro{\z}{\hh * -\level}
      \pgfmathsetmacro{\zz}{\hh * -(\level-1)}
      \draw[fill=gray!50] (\b,\z) rectangle (\d,\zz);
    }
    \foreach \b/\d [count=\level from 0] in {0/#1, 0/#2, 0/#4, #3/#4, #3/#5, #6/#5, #7/#5} {
      \pgfmathsetmacro{\z}{\hh * -\level}
      \draw (\b,\z)node[dotstyle]{} --node{{\ifthenelse{\level=3}{#8}{}}} (\d,\z)node[dotstyle]{};
    }
    \pgfmathsetmacro{\z}{\hh * -6}
    \pgfmathsetmacro{\zz}{\hh * -7}
  }
  \begin{tikzpicture}[scale=0.5]
    \tikzset{dotstyle/.style={circle,fill=black,inner sep=0pt,minimum size=3pt}}
    \begin{scope}[shift={(0,0)},scale=1]
      \drawprofile{2}{3}{1.5}{4}{6}{3}{4}{$A$}
    \end{scope}
    \draw[-stealth,red, thick] (6, -4) -- (6,-3)node[anchor=west,yshift=2pt]{$r_A$};
    \draw[-stealth,red, thick] (6, -4)node[dotstyle,fill=red]{}-- (7,-4) node[anchor=south,xshift=5pt]{$l_B$};
    \node [red,anchor=north east] at(6,-4){$x$};
    \begin{scope}[shift={(7,-4)},scale=1]
      \drawprofile{2}{3}{1.5}{4}{6}{3}{4}{$B$}
    \end{scope}
  \end{tikzpicture}.
\]
Note that the vector spaces at both $r_A$ and $l_B$ are $K$.
On the vertex $x$ below $r_A$ (and to the left of $l_B$), we associate the vector space $K$, and put identity maps on the arrows to $r_A$ and $l_B$. The concatenation $A\circ B$ is the $2$D persistence module defined as being equal to $A$ and $B$ restricted to their respective supports, $K$ over $x$, and with  identity maps over the arrows from $x$ to $l_B$ and $x$ to $r_A$, and zero elsewhere.

\begin{lemma}\label{lem:concatenation}
    Given $A$ and $B$ two candy modules, $A\circ B$ is a candy module.
\end{lemma}

\begin{proof}
  It is easy to see that $A\circ B$ is a well-defined $2$D persistence module, as no commutativity relations are violated by our construction.
  Indeed the supports of $A$ and $B$ are disjoint, and there are no paths from one to the other.
  Moreover, the addition of the vector space over $x$ and the two identity maps only create paths from $x$ to the upper row of $B$ and the right column of $A$.

  Next, we check that $A\circ B$ is indeed a candy module.
  % Every interval of $B$ is located on the right and below all the intervals of $A$, and thus the support of $A\circ B$ is contained in a staircase shape.
  Its top left vertex corresponds to the top left vertex of $A$ and is associated to the vector space $K$, while its bottom right vertex corresponds to the bottom right vertex of $B$ and also supports $K$.

  By the property of candy modules, $\End(A)\cong K\cong \End(B)$.
  Let $f$ be an endomorphism of $A\circ B$.
  The restriction of $f$ to the support of $A$, respectively $x$ and $B$, is defined by a unique scalar $\alpha$, respectively $\xi$, $\beta$.
  Therefore there exist two non-zero constants $c_A$ and $c_B$ such that $f$ restricted at vertex $r_A$, respectively $x$ and $l_B$ is the multiplication by the scalar $c_A\alpha$, respectively $\xi$ and $c_B\beta$.
  The commutativity conditions along the arrows from $x$ to $r_A$ and  from $x$ to $l_B$ implies that $c_A\alpha=\xi=c_B\beta$.
  Note that no other conditions are required for $f$.
  Therefore $f$ is completely determined by one scalar $\xi$, and the endomorphism ring of $A\circ B$ is isomorphic to $K$.
\end{proof}

\begin{lemma}
  \label{lem:countability}
  The set $\mathcal{P}$ of isomorphism classes of $1$D persistence modules of finite total dimension is a countable set.
\end{lemma}

\begin{proof}
  Note that a representative $V$ for a class $[V]\in \mathcal{P}$ can be decomposed as
  \[
    V \cong \bigoplus_{i=1}^m \intv[b_i,d_i]
  \]
  in a unique canonical way.
  Indeed the intervals in an indecomposable decomposition can be ordered using the lexicographic order on $\mathbb{Z}^2$.
  We thus have a unique representative for each class $[V]$ and denote by $S$ the set of all those representatives.

  We define $S_n$ to be the set of all elements of $S$ whose total dimension is less than $n$ and whose support is contained in $[-n,n]$.
  Every set $S_n$ is finite and $\cup_n S_n = S$.
  Therefore the set $S$ is countable and is in bijection with the set of all isomorphism classes of $1$D persistence modules of finite total dimension.

  Note that the sequence $(S_n)_n$ is in fact a filtration of $S$ as $S_n\subset S_{n+1}$.
  %\qed
  % We map $[V]$ to the finite length sequence $(b_1,d_1,b_2,d_2,\hdots, b_m,d_m)$ to create an map
  % $\mathcal{P} \rightarrow L$, where $L$ is the set of all finite length sequences, which is known to be countable. Note that $\mathcal{P} \rightarrow L$ is injective since $[V]$ is uniquely determined by its indecomposable decomposition. Thus the claim follows.
  %
  % First remark that due to the nature of line restrictions, we only need to consider $1$D persistence modules up to isomorphism and translation.
  % Let us fix $n>0$ and consider all possible $1$D persistence module with total dimension $n$.
  % Without loss of generality, we can translate them in order to see them as representation of $A_n$ with the extra condition that the vector space at the leftmost index is $0$.
  % The number of such modules is finite.
\end{proof}

\begin{theorem}
  There exists an indecomposable $2$D persistence module $M \in \rep\Gf{}$ such that any $1$D persistence module $V$ of finite total dimension is a line restriction of $M$.
\end{theorem}

\begin{proof}
    As a consequence of Lemma~\ref{lem:countability} the set of all non-zero $1$D persistence modules up to isomorphism and translation is a countable set. We can enumerate its elements as a sequence $(L_i)_{i\in \mathbb{N}}$.
    For each of these $L_i$, we apply our dual and primal construction to obtain a candy module $M_i$, as in Remark~\ref{rem:shape}.
    Note that every $M_i$ has exactly seven rows and that its endomorphism ring is $K$, see the proof of Theorem~\ref{thm:main} for details.

    We concatenate $M_0,M_1,\hdots$ to obtain a $2$D persistence module $M$ as below:
    \[
  \newcommand{\drawprofile}[8]{
    % d1, d2, bV, dV, d', b1, b2, label
    % intervals are:
    % (0,d1), (0,d2), (0,dV)    [dual]
    % (bV,dV)                   [V]
    % (bV,d'), (b1,d'), (b2,d') [primal]
    \def\hh{0.5}
    \foreach \b/\d [count=\level from 0] in {0/#1, 0/#2, #3/#4}{
      \pgfmathsetmacro{\z}{\hh * -\level}
      \pgfmathsetmacro{\zz}{\hh * -(\level+1)}
      \draw[fill=gray!50] (\b,\z) rectangle (\d,\zz);
    }
    \foreach \b/\d [count=\level from 4] in {#3/#4, #6/#5, #7/#5} {
      \pgfmathsetmacro{\z}{\hh * -\level}
      \pgfmathsetmacro{\zz}{\hh * -(\level-1)}
      \draw[fill=gray!50] (\b,\z) rectangle (\d,\zz);
    }
    \foreach \b/\d [count=\level from 0] in {0/#1, 0/#2, 0/#4, #3/#4, #3/#5, #6/#5, #7/#5} {
      \pgfmathsetmacro{\z}{\hh * -\level}
      \draw (\b,\z)node[dotstyle]{} --node{{\ifthenelse{\level=3}{#8}{}}} (\d,\z)node[dotstyle]{};
    }
    \pgfmathsetmacro{\z}{\hh * -6}
    \pgfmathsetmacro{\zz}{\hh * -7}
    \pgfmathsetmacro{\xx}{#5 + \hh}
    \draw[-stealth,red, thick] (#5, \zz) -- (#5, \z);
    \draw[-stealth,red, thick] (#5, \zz) -- (\xx, \zz);
  }
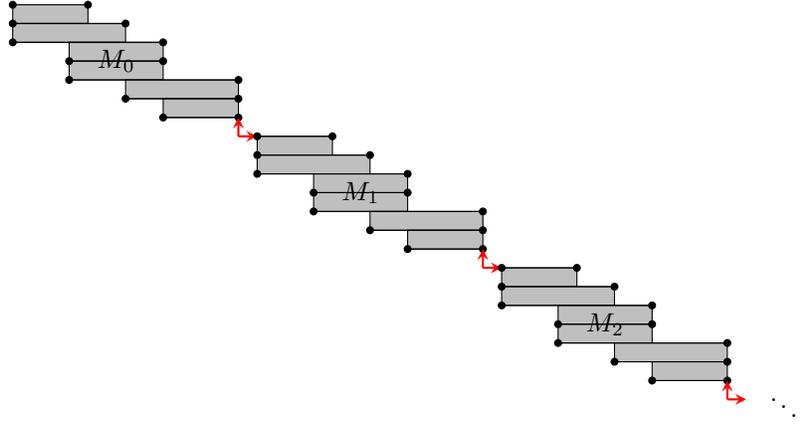
  \begin{tikzpicture}[scale=0.5]
    \tikzset{dotstyle/.style={circle,fill=black,inner sep=0pt,minimum size=3pt}}
    \begin{scope}[shift={(0,0)},scale=1]
      \drawprofile{2}{3}{1.5}{4}{6}{3}{4}{$M_0$}
    \end{scope}
    \begin{scope}[shift={(6.5,-3.5)},scale=1]
      \drawprofile{2}{3}{1.5}{4}{6}{3}{4}{$M_1$}
    \end{scope}
    \begin{scope}[shift={(13,-7)},scale=1]
      \drawprofile{2}{3}{1.5}{4}{6}{3}{4}{$M_2$}
    \end{scope}
    \begin{scope}[shift={(19.5,-10.5)},scale=1]
      \node at (1,0) {$\ddots$};
    \end{scope}
  \end{tikzpicture}.
\]
    The proof of Lemma~\ref{lem:concatenation} naturally extends to the setting of a countable concatenation, and therefore $M$ is a well-defined persistence module with $\End(M)\cong K$.
    Thus $M$ is an indecomposable $2$D persistence module.

    Note that the line just above the module $M_0$ is $0$ so $M$ also contains the zero $1$D persistence module as a line restriction.
    By construction all non-zero $1$D persistence modules with finite total dimension appear as a line restriction to the middle row of the corresponding candy module $M_i$.
    Every vertex supports a finite dimensional vector space so $M$ is locally finite. Obviously $M$ has infinite support and vector spaces with arbitrarily high dimension.
    %\qed
\end{proof}

%%% Local Variables:
%%% mode: latex
%%% TeX-master: "arxiv_main"
%%% End:

\section{Rectangle-decomposable $n$D persistence modules}
\label{sec:rectangles}

Denote elements of $\mathbb{Z}^n$ as $n$-tuples $\vv{b} = (b^1,b^2,\hdots,b^n)\in\mathbb{Z}^n$.
We consider the poset $(\mathbb{Z}^n,\leq)$, where $\vv{b} \leq \vv{d}$ if and only if $b^k \leq d^k$ for all $k=1,\hdots,n$.

Recall that by $n$D persistence module we mean a functor $M$ in $\rep \mathbb{Z}^n = \left[\mathbb{Z}^n, \vect_K\right]$. Throughout this section, persistence module will mean persistence module of finite support: $\{M(\vv{x}) \neq 0 \suchthat \vv{x} \in \mathbb{Z}^n\}$ is a finite set.

\begin{definition}
  Let $\vv{b} \leq \vv{d} \in \mathbb{Z}^n$. The \emph{finite rectangle} from $\vv{b}$ to $\vv{d}$, denoted $\intv[\vv{b},\vv{d}]$, is the $n$D persistence module defined by
  \[
    \begin{array}{rcl}
    \intv[\vv{b},\vv{d}](\vv{x}) &=&
    \left\{
      \begin{array}{ll}
        K & \text{ if } \vv{b} \leq \vv{x} \leq \vv{d}, \\
        0 & \text{ otherwise, and}
      \end{array}
    \right. \\~\\
    \intv[\vv{b},\vv{d}](\vv{x} \leq \vv{y}) &=&
    \left\{
      \begin{array}{ll}
        1_K & \text{ if } \vv{b} \leq \vv{x} \leq \vv{y} \leq \vv{d}, \\
        0 & \text{ otherwise}.
      \end{array}
            \right.
    \end{array}
  \]
\end{definition}
The support of $\intv[\vv{b},\vv{d}]$ is exactly the points on the $n$D grid  contained in the rectangle
\[
  [b^1,d^1] \times [b^2,d^2] \times \hdots \times [b^n,d^n].
\]
We denote this set of points on the $n$D grid by $[\vv{b},\vv{d}]$.

A more general treatment would involve possibly unbounded rectangles (See \cite{bjerkevik2016stability} for example). However, we only treat $n$D persistence modules of finite support, and this finiteness is essential to our proof of Theorem~\ref{thm:rectangles}.

\begin{definition}
  An $n$D persistence module $V$ is said to be finite-rectangle-decomposable if and only if
  \[
    V \cong \bigoplus_{i=1}^m\intv[\vv{b_i}, \vv{d_i}].
  \]
  for some elements $\vv{b_i} \leq \vv{d_i} \in \mathbb{Z}^n$ for $i=1,\hdots m$.
\end{definition}
Note that we assume finite supports, and only discuss finite sums. The finite rectangles $\intv[\vv{b_i}, \vv{d_i}]$ do not need to be distinct.

The following lemma generalizes Lemma~\ref{lem:homdim} to finite rectangles.
\begin{lemma}
  \label{lem:homdim_rectangles}
  Let $\intv[\vv{a},\vv{b}],\intv[\vv{c},\vv{d}]$ be finite rectangular $n$D persistence modules.
  \begin{enumerate}
  \item The dimension of $\Hom(\intv[\vv{a}, \vv{b}], \intv[\vv{c}, \vv{d}])$ as a $K$-vector space is either $0$ or $1$.
  \item There exists a canonical basis $\left\{f\itoi{\vv{a},\vv{b},\vv{c},\vv{d}}\right\}$ for each nonzero $\Hom(\intv[\vv{a}, \vv{b}], \intv[\vv{c}, \vv{d}])$ such that
    \[
      (f\itoi{\vv{a},\vv{b},\vv{c},\vv{d}})_{\vv{x}} = \left\{
        \begin{array}{ll}
          1_K: K \rightarrow K, & \text{if } \vv{x} \in [\vv{a},\vv{b}] \cap [\vv{c},\vv{d}] \\
          0, & \text{otherwise.}
        \end{array}
      \right.
    \]
    \item  In fact,
    \[
      \Hom(\intv[\vv{a}, \vv{b}], \intv[\vv{c}, \vv{d}])=
      \begin{cases}
        K f\itoi{\vv{a},\vv{b},\vv{c},\vv{d}}, & \vv{c} \leq \vv{a} \leq \vv{d} \leq \vv{b}, \\
        0, & \text{otherwise}.
      \end{cases}
    \]
  \end{enumerate}
\end{lemma}
\begin{proof}
  Lemma~3.3 and Corollary~3.4 of \cite{bjerkevik2016stability}.
  %\qed
\end{proof}

Thus, for morphisms between finite-rectangle-decomposable $n$D persistence modules, we have a similar matrix formalism as in the $1$D case. As before, we adopt the convention of hiding the morphisms  $f\itoi{\vv{a},\vv{b},\vv{c},\vv{d}}$ for simplicity.

% That is, for $\phi:V\rightarrow W$ where $V = \bigoplus_{i=1}^m\intv[\vv{b_i}, \vv{d_i}]$ and $W = \bigoplus_{i=1}^{m'}\intv[\vv{x_i}, \vv{y_i}]$, ...

\begin{definition}[Vertical]
  A collection of finite rectangles $\mathcal{I} = \{\intv[\vv{b_i},\vv{d_i}]\}_{i=1}^m$ is said to be \emph{vertical} if and only if all the pairs $(\vv{b_i},\vv{d_i})$ are distinct and there exists a constant $\vv{\mu}$ such that $\vv{\mu} = \frac{\vv{b_i}+\vv{d_i}}{2}$ for all $i$.
\end{definition}
In the definition above, the arithmetic operations are performed component-wise. The following is the generalized version of Lemma~\ref{lem:vertical}.

\begin{lemma}
  Let $W = \bigoplus_{i=1}^m \intv[\vv{b_i},\vv{d_i}]$ be finite-rectangle-decomposable $n$D persistence module whose set of rectangles is vertical. Then the matrix form of any morphism $\phi: W\rightarrow W$ is diagonal:
  \[
    \phi =
    \left[
      \begin{array}{ccc}
        c_1 & & \\
            & \ddots & \\
            & & c_m
      \end{array}
    \right]
  \]
  for some scalars $c_1,c_2,\hdots, c_m \in K$.
\end{lemma}
\begin{proof}
  Similar to the proof of Lemma~\ref{lem:vertical}, and using Lemma~\ref{lem:homdim_rectangles}.
  %\qed
\end{proof}

\begin{definition}[Hyperplane Restriction]
  \label{defn:hyperplane_restriction}
  Let $n$ be a positive integer.
  \begin{enumerate}
  \item A \emph{hyperplane} $L$ is a functor $L : (\mathbb{Z}^n,\leq) \rightarrow (\mathbb{Z}^{n+1},\leq)$ that is injective on objects.
  \item Let $V$ be an $n$D persistence module. We say that $V$ is a \emph{hyperplane restriction} of the $(n+1)$D persistence module $M$ if there is a hyperplane $L$ such that $(M\circ L) \cong V$.
  \end{enumerate}
\end{definition}

Note that in the proof of Theorem~\ref{thm:main}, we used the Auslander-Reiten quiver of $\Af{w}$ in order to visualize the movements of the intervals. However, it was not essential to the proof. Above, we have provided generalizations of the essential tools to finite rectangles. Thus, the following theorem can be proved similarly as Theorem~\ref{thm:main}, except we no longer draw the Auslander-Reiten quiver.
\begin{theorem}
  \label{thm:rectangles}
  Let $V$ be a finite-rectangle-decomposable $n$D persistence module. Then there exists an indecomposable $(n+1)$D persistence module $M$ with finite support such that $V$ is a hyperplane restriction of $M$.
\end{theorem}
\begin{proof}
  The proof is similar to the proof of Theorem~\ref{thm:main}. Below is a quick outline.

  Without loss of generality, let $V = \bigoplus_{i=1}^m \intv[\vv{b_i},\vv{d_i}]$. We then separate-and-shift to distinct $\vv{d'_i} \in \mathbb{Z}^n$ such that $\vv{d'_i} \geq \vv{d_i}$ for all $i$ and $\frac{\vv{b_j} + \vv{d'_j}}{2} \leq \vv{d'_i}$ for all pairs $i$ and $j$. Take
  \[
    \vv{\mu} = \max_{i}\left(\frac{\vv{b_i} + \vv{d'_i}}{2}\right)
  \]
  where the $\max$ is taken component-wise. This will serve as the shared midpoint for verticalization, by defining $\vv{b_i'} = 2\vv{\mu} - \vv{d'_i}$. Again, we have
  \[
    \vv{b_i} \leq \vv{b'_i} \leq \vv{\mu} \leq \vv{d'_j}
  \]
  for all pairs $i$, $j$. Finally, coning is done with the finite rectangle
  \[
    I_V \defeq \intv\left[\max_i\left(\vv{b'_i}\right), \max_j\left(\vv{d'_j}\right)\right].
  \]

  We get the object $S\in \fun{\Af{4}}{\rep(\mathbb{Z}^n,\leq)}$, a $\rep(\mathbb{Z}^n,\leq)$-valued representation of $\Af{4}$:
  \[
    S:
    \begin{tikzcd}[ampersand replacement=\&, column sep=large]
      I_V \rar{\smat{1\\\vdots\\1}} \&
      \widebar{V} \rar{\smat{1 & & \\ & \ddots & \\ & & 1}} \&
      V' \rar{\smat{1 & & \\ & \ddots & \\ & & 1}} \&
      V
    \end{tikzcd}
  \]
  where $S(3)\defeq V$, $S(2) \defeq V'  = \bigoplus_{i=1}^m \intv[\vv{b_i},\vv{d'_i}]$ is obtained by separate-and-shift, $S(1) \defeq \widebar{V} = \bigoplus_{i=1}^m \intv[\vv{b'_i},\vv{d'_i}]$ by verticalization (of $V'$), and $S(0) \defeq I_V$ by coning. By stacking, $S$ can be viewed as a representation of $(\mathbb{Z}^{n+1},\leq)$, id est, an $(n+1)$D persistence module. Note that the $(n+1)$D persistence module obtained from $S$ clearly has finite support. The original $n$D persistence module $V$ is clearly a hyperplane restriction by a natural inclusion $\mathbb{Z}^n \hookrightarrow \mathbb{Z}^{n+1}$.

  It can be shown that $\End(S) \cong K$, proving the indecomposability of $S$. Note that a generalization of Lemma~\ref{lem:prop} (propagation of nonzeros) to finite-rectangle-decomposables is needed, but the proof is similar to that of Lemma~\ref{lem:prop}.
  %\qed
\end{proof}

Note that, as before, the construction in the proof of the theorem has a dual, which we denote by $
	V \rightarrow
    V'' \rightarrow
    \widebar{\widebar{V}} \rightarrow
    I'_V
$.
A direct consequence of this extension is the construction of indecomposable $(n+1)$D persistence modules whose support has an arbitrary number of $n$D holes. Here, $n$D holes in the support can be formalized as linearly independent classes of the $n$-dimensional homology of the clique cubical complex of the support.

\begin{corollary}\label{cor:hyperholes}
    For any $l\in\{0,1,2,\dots\}$, there exists an indecomposable $(n+1)$D persistence module whose support has exactly $l$ $n$D holes and no lower dimensional hole.
\end{corollary}

\begin{proof}
The proof is a direct adaptation of the proof of Corollary~\ref{cor:holes}.
%
% Moreover the persistence module $S:I_V\rightarrow\bar V\rightarrow V'\rightarrow V$ has a support with contractible cubical complex.
% The same holds for the dual.
% Finally, i

Clearly, it is possible to build an $n$D persistence module $V$ that is rectangle decomposable and contains exactly $l$ $(n-1)$D holes in its support and no lower dimensional hole, for example by building with rectangles the boundaries of $l$ $n$D hypercubes.

Note that the construction of Theorem~\ref{thm:rectangles} can be dualized in the same way as the one of Theorem~\ref{thm:main}. As in the proof of Corollary~\ref{cor:holes}, we obtain the $(n+1)$D persistence module $H$ obtained by combining the primal and dual construction with $V$ in the middle. The $(n+1)$D persistence module $H$ has support with exactly $l$ $n$D holes and is indecomposable.

To see this, by considering the construction as working in the up and down direction, we can interpret it as stacking $n$D slices: $I_V$, $\bar V$, and $V'$ below $V$ for the primal construction, and $V''$, $\widebar{\widebar{V}}$, and $I'_V$ above $V$ for the dual construction. Each of those new slices has contractible support, and they are connected along a contractible intersection, so that $I_V\rightarrow\bar V\rightarrow V'$ has contractible support (and similarly for the dual part).

The construction with primal and dual parts modifies the support to become a double cone with apexes $I_V$ and $I'_V$, in other words its topological suspension.
Such operation transforms each $m$D hole into an $(m+1)$D hole for every $m$, without creating other topological features. In our case this results in the requested topology.
%\qed
\end{proof}

Gluing the constructions in all dimensions using an operation similar to the concatenation of Section~\ref{sec:lotr}, we obtain the following more interesting result.

\begin{theorem}
    Given a positive dimension $n$ and a sequence $(b_i)_{0<i<n}$ of non-negative integers, there exists an indecomposable $n$D persistence module whose support has $b_i$ for its Betti number for every $0<i<n$.
\end{theorem}

\begin{proof}
  Using Corollary~\ref{cor:hyperholes}, we first build a sequence of indecomposable  persistence modules $T_i$ for $0<i<n$ that satisfy the condition for exactly the $i$th Betti number.
  % More precisely, $T_i$ is the indecomposable $(i+1)$D persistence module obtained by starting from a set of $i$-dimensional hypercubes forming $b_i$ $(i-1)$-dimensional holes and applying both primal and dual constructions to it.
  More precisely, each $T_i$ is an indecomposable $(i+1)$D persistence module whose support has exactly $b_i$ $i$-dimensional holes and no other holes. In other words, its Betti numbers are $1$ in dimension $0$, $b_i$ in dimension $i$ and $0$ otherwise.

  By construction, $T_i$ has two other important properties.
  First, $\End(T_i) \cong K$. Second, the support of $T_i$ has two extremal vertices, along which we string together these persistence modules. These extremal vertices are minimal and maximal under the following partial order: the usual order for each of the first $i$ coordinates and the opposite order for the last coordinate (along which we stack). In symbols, we define
  $(x^1,\hdots,x^i,x^{i+1}) \leq (y^1,\hdots,y^i,y^{i+1})$ if and only if $x^j \leq y^j$ for all $1\leq j \leq i$ and $x^{i+1} \geq y^{i+1}$. For example, $(1,2,2)<(2,3,1)$.
  Under this order and by construction, the support of $T_i$ has a minimal element $l_i$ and a maximal element $r_i$ and both of them host a one-dimensional vector space.
  Note that this order is only introduced to simplify the specification of these extremal points. We are not changing the underlying poset of which $T_i$ is a representation.

    % Second, the support of $T_i$ has two extremal vertices which are minimal and maximal in the following sense.
    % We consider coordinates $1$ to $i$ increasingly and the last coordinate decreasingly to define a partial order on the indices.
    % In other words $(1,2,2)<(2,3,1)$ for example.
    % Then by construction, the set of indices has a minimal element $l_i$ and a maximal element $r_i$ and both of them host a one-dimensional vector space.
    % Note that this order is only introduced to simplify the use of the minimality and maximality condition but does not interfere with the structure of the grid.

    We now build an indecomposable persistence module $M$ with the required property inductively.
    Let $M_1=T_1$. Clearly, $M_1$ is an indecomposable $2$D persistence module with $\End(M)\cong K$ and whose support has Betti numbers $1, b_1$. Furthermore, $M_1(r_1) = K$ (a one-dimensional $K$-vector space), by construction.

    Assume that we have built $M_i$, an indecomposable $(i+1)$D persistence module with $\End(M_i) \cong K$, Betti numbers $1,b_1,\dots,b_i$, and with maximal element $r_i$ in its support where $M_i(r_i) = K$.

    We embed $M_i$ into the $(i+1)+1 = (i+2)$-dimensional grid by adding an extra $0$ coordinate at the end of all indices, and denote it by $M_i'$.
    We translate the $(i+2)$D persistence module $T_{i+1}$ in the grid such that the maximal element $r_i'$ (which is $r_i$ after embedding) and $l_{i+1}$ (after this translation) are adjacent on the $(i+2)$D grid, with the only difference coming from the last coordinate.

    The $(i+2)$D persistence module $M_{i+1}$ is defined to be equal to $M'_i$ over its support and equal to the translated $T_{i+1}$ over its support, with the identity linear map for the arrow from $M_{i+1}(r'_i) = M'_i(r'_i) = K$ to  $M_{i+1}(l_{i+1}) = T_{i+1}(l_{i+1}) = K$, and zero elsewhere.

    First, $M_{i+1}$ is obtained by assembling the two modules $M_i'$ and $T_{i+1}$.
    To see that this is indeed a $(i+2)$D persistence module, we check that no nontrivial commutativity condition exists between the two parts as no pair of vertices in the support of $M'_i$ and $T_{i+1}$ are comparable except for those on the same slice as $l_{i+1}$.
    In other words, no path in the grid, from the support of $M'_i$ to the support of $T_{i+1}$, exists that does not go through the newly added arrow.

    The support of $M_{i+1}$ is obtained by linking the support of $M_i'$ and $T_{i+1}$ by a single arrow.
    The Betti numbers of $M_{i+1}$ are therefore the sum of the Betti numbers of $M_i'$ and $T_{i+1}$ for all positive dimensions and the sum minus one for $\beta_0$.
    Hence, $M_{i+1}$ has Betti numbers $1,b_1,\dots,b_{i+1}$.

    The shifted index $r_{i+1}$ is maximal for $M_{i+1}$.

    The endomorphism ring of $M_{i+1}$ is isomorphic to $K$.
    Indeed, let $f$ be an endomorphism of $M_{i+1}$.
    By hypotheses, $f$ is uniquely defined by two scalar values $\alpha$ and $\beta$ corresponding to the parts $M'_i$ and $T_{i+1}$.
    However, we have an extra condition due to the arrow linking the two parts of the module.
    An elementary computation shows that this condition implies that $\alpha=c\beta$ for some constant $c$, and therefore every endomorphism is uniquely defined by one scalar.

    By induction $M=M_n$ satisfies all the desired conditions and therefore is an indecomposable $(n+1)$D persistence module whose support has Betti numbers $1,b_1,\dots,b_n$.
    %\qed
\end{proof}

%%% Local Variables:
%%% mode: latex
%%% TeX-master: "arxiv_main"
%%% End:

\section{Discussion}

The construction presented in this paper shows, for any given $1$D persistence module with finite support, how to build a $2$D indecomposable module containing it as a line restriction.
In fact we also built a single $2$D indecomposable with infinite support that contains all $1$D persistence modules with finite support as line restrictions.

We then showed that our main construction naturally extends to any finite-rectangle-decomposable $n$D persistence module, showing that each can be found in some indecomposable $(n+1)$D persistence module as a hyperplane restriction.
As a side result, we were also able, given arbitrary numbers $\beta_l$ for $1\leq l\leq n-1$, to build an indecomposable $n$D persistence module whose cubical complex has Betti number $\beta_l$ for all dimensions $1\leq l\leq n-1$.

Our construction can be extended to persistence modules indexed over $\mathbb{R}$ that can be decomposed into a finite direct sum of finite intervals. This extension requires a careful handling of the various types of intervals (open, closed and half-open), and separate-and-shift requires an extra enlarging of death indices (the second condition should be a strict inequality).

It is natural to ask whether or not our construction can be made functorial. In particular, suppose that for each $V \in \rep\Afn$, we fix exactly one construction $S(V)$ as in Eq.~\eqref{eq:master_construction} (recall that there are many choices in the construction; fix one choice for each).
If we let $L$ be the horizontal line at index $3$ ($4$th row), clearly $R_L(S(V)) = V$, where $V$ is viewed as an object in $\rep \mathbb{Z}$. Indeed, this is what we proved in Theorem~\ref{thm:main}.

Then the question is as follows. For each morphism $f:V\rightarrow W$, can we give a morphism $S(f) : S(V) \rightarrow S(W)$ so that $S:\rep\Afn \rightarrow \rep G$ becomes a functor? More strictly, is it possible to give the values $S(f)$ so that $R_L \circ S$ is (or at least isomorphic to) the identity functor? Unfortunately, it seems like our construction cannot be made into a functor in such a nice way.

Our methods relying on the structure of $1$D interval ($n$D finite-rectangle) representations do not extend trivially to general $n$D persistence modules. We do not yet know whether or not statements similar to our main Theorem~\ref{thm:main} and Theorem~\ref{thm:rectangles} can be made for general $n$D persistence modules, or indeed, for interval-decomposable $n$D persistence modules with $n > 1$.

%%% Local Variables:
%%% mode: latex
%%% TeX-master: "arxiv_main"
%%% End:

\section*{Acknowledgements}
  We would like to thank Peter Bubenik for suggesting the extension of our main result to finite-rectangle-decomposable $n$D persistence modules, and Tomoo Yokoyama for raising the question which we studied in Section~\ref{sec:lotr}.
\bibliographystyle{plain}
\bibliography{refs}

\appendix
\section{Appendix}
\label{appendix}
Here, we reproduce Lemma~\ref{lem:vertical_stacking} and Lemma~\ref{lem:cross} and provide their proofs.

\noindent\parbox{\textwidth}{\stackinglem*}

\begin{proof}
  Let us construct an isomorphism
  \[
    \Phi: \rep \Gf{h,w} \rightarrow \fun{\Af{h}}{\rep\Af{w}}
  \]
  with inverse
  \[
    \Psi: \fun{\Af{h}}{\rep\Af{w}} \rightarrow \rep \Gf{h,w}.
  \]

  \paragraph{Effect of $\Phi$ on objects.}
  For $V \in \rep \Gf{h,w}$, define the representation $\Phi(V): \Af{h} \rightarrow \rep\Af{w}$ by the following. On $j\in [h]$, we set
  \[
    \Phi(V)(j) = V(j,-)
  \]
  which can be checked to be a functor from $\Af{w}$ to $\vect_K$, and is thus an object in $\rep\Af{w}$.
  This is the restriction of $V$ to row $j$.
  On $i\leq j$ in $[h]$, we set
  \[
    \Phi(V)(i\leq j) \defeq V((i,-) \leq (j,-)): V(i,-) \rightarrow V(j,-),
  \]
  which is the collection of maps indexed by $[w]$ with value at $k\in[w]$:
  \[
    \Phi(V)(i\leq j)_k = V((i,k)\leq(j,k)): V(i,k) \rightarrow V(j,k).
  \]
  That each $\Phi(V)(i\leq j)$ is a morphism (a natural transformation) in $\rep\Af{w}$ follows from the functoriality of $V$. In detail, we have for each $k\leq \ell$,
  \[
    \begin{tikzcd}[column sep=6em]
      V(i,-) \dar{\Phi(V)(i\leq j)} &
      V(i,k)) \rar{\gamma \defeq V((i,k)\leq(i,\ell))}  \dar{\alpha\defeq V((i,k)\leq(j,k))} &
      V(i,\ell) \dar{\delta\defeq V((i,\ell)\leq(j,\ell))}  \\
      V(j,-) &
      V(j,k) \rar[swap]{\beta\defeq V((j,k)\leq(j,\ell))} & V(j,\ell)
    \end{tikzcd}
  \]
  commutes since $\beta\alpha = V((i,k)\leq (j,\ell)) = \delta\gamma$.

  Finally, it can be checked that $\Phi(V)$ itself is a functor $\Phi(V): \Af{h}\rightarrow \rep\Af{w}$, and thus $\Phi(V)$ is an object of $\fun{\Af{h}}{\rep\Af{w}}$.

  \paragraph{Effect of $\Phi$ on morphisms.}
  For $f:V\rightarrow W$ a morphism in $\rep \Gf{h,w}$, let us construct $\Phi(f):\Phi(V)\rightarrow \Phi(W)$. This is a collection of maps $\Phi(f)_j:  \Phi(V)(j) \rightarrow \Phi(W)(j)$ for $j\in [h]$, where $\Phi(f)_j$ is given by restriction to the $j$th row:
  \[
    \Phi(f)_j \defeq f_{(j,-)}: V(j,-) \rightarrow W(j,-).
  \]
  This $j$th row is the collection $\Phi(f)_j = \{(\Phi(f)_j)_i\}_{i\in[w]}$ with
  \[
    (\Phi(f)_j)_i \defeq f_{(j,i)} : V(j,i) \rightarrow W(j,i).
  \]
  For each $j$, that $\Phi(f)_j$ is a natural transformation follows from naturality of $f$.

  Then we need to check that the collection $\Phi(f) = \{\Phi(f)_j\}$ in fact defines a natural transformation. This also follows from naturality of $f$.

  \paragraph{Functoriality of $\Phi$.}
  Next, $\Phi$ is clearly a functor: $\Phi(1)_j = 1_{(j,-)}$ so that $\Phi(1)$ is the identity, and
  \[
    (\Phi(gf))_j = (gf)_{(j,-)} = g_{(j,-)}f_{(j,-)} = \Phi(g)_j \Phi(f)_j.
  \]

  \paragraph{Effect of $\Psi$ on objects.}
  In the other direction, let now $M \in \fun{\Af{h}}{\rep\Af{w}}$. Let us define $\Psi(M) \in \rep \Gf{h,w}$.
  We put $\Psi(M)(i,j) = M(i)(j)$ for $(i,j) \in \Gf{h,w}$. For the pairs $(i,j)\leq (k,\ell)$,
  \[
    \Psi(M)((i,j)\leq (k,\ell)): \Psi(M)(i,j) \rightarrow \Psi(M)(k,\ell),
  \]
  should be a morphism from $M(i)(j)$ to $M(k)(\ell)$, which we define to be either composition in
  \[
    \begin{tikzcd}[column sep=3em]
      M(i)(j) \rar{M(i)(j\leq \ell)}\dar[swap]{M(i\leq k)_j} &
      M(i)(\ell) \dar{M(i\leq k)_\ell} \\
      M(k)(j) \rar{M(k)(j\leq \ell)} & M(k)(\ell)
    \end{tikzcd}.
  \]
  It does not matter which path is taken, as they are equal by naturality of $M(i\leq k): M(i)(-) \rightarrow M(k)(-)$.

  Then it can be checked that $\Psi(M):\Gf{h,w} \rightarrow \vect_K$ is a functor, and thus an object in $\rep \Gf{h,w}$.

 \paragraph{Effect of $\Psi$ on morphisms.}   For $g:M\rightarrow N$ a morphism in $\fun{\Af{h}}{\rep\Af{w}}$,
  let $\Psi(g)_{(i,j)} = (g_i)_j$ for $(i,j) \in \Gf{h,w}$. That $\Psi(g):\Psi(M) \rightarrow \Psi(N)$ is a morphism can be checked by using naturality of $g$, as follows. For $(i,j) \leq (k,\ell)$ in $\Gf{h,w}$, consider the diagram
  \begin{equation}
    \label{diag:psi_morphism}
    \begin{tikzcd}[column sep=3em]
      M(i)(j) \rar{M(i)(j\leq \ell)}\dar{(g_i)_j} & M(i)(\ell) \rar{M(i\leq k)_\ell}\dar{(g_i)_\ell} & M(k)(\ell)\dar{(g_k)_\ell} \\
      N(i)(j) \rar{N(i)(j\leq \ell)} & N(i)(\ell) \rar{N(i\leq k)_\ell} & N(k)(\ell)
    \end{tikzcd}
  \end{equation}
  where the composition of the top row is $\Psi(M)((i,j)\leq (k,\ell))$ and the bottom row is $\Psi(N)((i,j)\leq (k,\ell))$. The commutativity of the left square follows from naturality of $g_i$ for fixed $i$, and that of the right square follows from naturality of $g$ and then evaluating at $\ell$. This shows the commutativity of Diagram~\eqref{diag:psi_morphism}, and thus $\Psi(g)$ is indeed a natural transformation.

\paragraph{Functoriality of $\Psi$.}  $\Psi(1)$ is clearly the identity, and
  \[
    \Psi(gf)_{(i,j)} = ((gf)_{i})_j = (g_i f_i)_j = (g_i)_j (f_i)_j = \Psi(g)_{(i,j)}\Psi(f)_{(i,j)}.
  \]

\paragraph{Inverse.}
  By the definitions given
  \[
	\Psi(\Phi(V))(i,j) = \Phi(V)(i)(j) = V(i,-)(j) = V(i,j)
  \]
  and $\Psi(\Phi(V))((i,j)\leq(k,\ell))$ is given by the composition of the top row in
  \[
    \begin{tikzcd}[column sep=large]
      \Phi(V)(i)(j) \rar{\Phi(V)(i)(j\leq \ell)}\dar[equal] &
      \Phi(V)(i)(\ell) \rar{\Phi(V)(i\leq k)_\ell} \dar[equal]&
      \Phi(V)(k)(\ell) \dar[equal]\\
      V(i,j) \rar{V((i,j)\leq(i,\ell))} &
      V(i,\ell) \rar{V((i,\ell)\leq (k,\ell))} &
      V(k,\ell)
    \end{tikzcd}.
  \]
  By functoriality of $V$, we see from the above commutative diagram that
  \[
    \Psi(\Phi(V))((i,j)\leq(k,\ell)) = V((i,j)\leq(k,\ell)).
  \]
  Thus, $\Psi\Phi(V) = V$ as functors.

  On morphisms $f:V\rightarrow W$ in $\rep \Gf{h,w}$, we see that
  \[
    \Psi(\Phi(f))_{(i,j)} = (\Phi(f)_i)_j = (f_{(i,-)})_j = f_{(i,j)},
  \]
  and so $\Psi\Phi(f) = f$. Thus, $\Psi\Phi$ is equal to the identity functor.

  In the other direction, it can be checked that $\Phi(\Psi(M)) = M$ for functors $M\in \fun{\Af{h}}{\rep\Af{w}}$, and similarly $\Phi(\Psi(f)) = f$ for morphisms $f$. Thus, $\Phi\Psi$ is also an identity functor.
  %\qed
\end{proof}

\noindent\parbox{\textwidth}{\crosslem*}

\begin{proof}
  Write $T \in \mathcal{P}$ as:
	\[
		T:
		\begin{tikzcd}[ampersand replacement=\&]
			\& D \& \\
			B \rar{g} \& E \uar{i} \rar{h} \& C \\
			\& A \uar{f} \&
		\end{tikzcd}
	\]
	with the relations $hf=0$ and $ig=0$ which entails $\Ima(f)\subset\Ker(h)$ and $\Ima(g)\subset\Ker(i)$.

	Using the isomorphism theorems, we can decompose $A$ as:
	\begin{align*}
		A \cong & A/\Ker(if) \bigoplus \Ker(if)/f^{-1}(\Ima(g)) \bigoplus f^{-1}(\Ima(g))/\Ker(f) \bigoplus \Ker(f).
	\end{align*}
	A similar decomposition can be obtained for $B$ due to the symmetry of the structure.
	Similarly, we can decompose $C$ as:
	\begin{align*}
		C \cong & \Ima(hg) \bigoplus h(\Ker(i))/\Ima(hg) \bigoplus \Ima(h)/h(\Ker(i)) \bigoplus C/\Ima(h)
	\end{align*}
  and symmetrically $D$.
	The middle vector space decomposes into $9$ summands:
	\begin{align*}
		E \cong & E/(\Ker(i)+\Ker(h)) \bigoplus (\Ker(i)+\Ker(h)) \\
    \cong & E/(\Ker(i)+\Ker(h)) \bigoplus \Ker(i)/(\Ker(i)\cap\Ker(h))\\ & \bigoplus \Ker(h)/(\Ker(i)\cap\Ker(h)) \bigoplus \Ker(i)\cap\Ker(h) \\
    \cong & E/(\Ker(i)+\Ker(h)) \\
            & \bigoplus \left[\Ima(g)/(\Ima(g)\cap\Ker(h)) \bigoplus \Ker(i)/(\Ima(g)+(\Ker(i)\cap\Ker(h)))\right] \\
            & \bigoplus \left[ \Ima(f)/(\Ima(f)\cap\Ker(i)) \bigoplus \Ker(h)/(\Ima(f)+(\Ker(i)\cap\Ker(h)))\right] \\
		        & \bigoplus \left[ \Ima(g)\cap\Ima(f) \bigoplus (\Ker(i)\cap\Ima(f))/(\Ima(f)\cap\Ima(g))\right. \\
                &\bigoplus (\Ima(g)\cap\Ker(h))/(\Ima(f)\cap\Ima(g)) \\
            & \left.\bigoplus (\Ker(i)\cap \Ker(h))/((\Ima(f)+\Ima(g))\cap\Ker(i)\cap\Ker(h))\right].
	\end{align*}

	Note that the maps induced by $f$, $g$, $h$ and $i$ on these decompositions restricted to one of the summands in its domain has image fully contained in at most one summand of its codomain. Furthermore, every summand in the codomain contains the image of at most one summand of the domain.
	In other words, this decomposition of the vector spaces induces a direct sum decomposition of $T$. Note that we did not necessarily compute an indecomposable decomposition of $T$, as this decomposition already suffices for our purposes.

	Finally, inspecting the form of the decomposition gives the claimed result.
  %\qed
\end{proof}

\end{document}

%%% Local Variables:
%%% mode: latex
%%% TeX-master: t
%%% End: